\documentclass[11pt]{amsart}
\usepackage{amsmath}
\usepackage{mathrsfs}
\usepackage{mathtools}
\usepackage{psfrag}
\usepackage{xcolor}
\usepackage[cmtip,line,color,poly,all]{xy}
\usepackage{graphicx,amssymb,booktabs, verbatim, mathtools}
\usepackage[noadjust]{cite}
\usepackage{xspace}
\usepackage{blindtext}
\usepackage{tikz}
\usepackage{longtable}
\usepackage{pdflscape}
\usepackage{esvect}

\usepackage[pdfusetitle, pdfpagelabels, plainpages=false,
  bookmarks, bookmarksnumbered
]{hyperref} 
\usepackage{bookmark}
\usepackage{subfigure}
\usepackage[margin=1.0in]{geometry}

\newif\iffloattoend

\iffloattoend
\usepackage[figuresonly, nomarkers, nolists, heads]{endfloat}

\fi
\newif\ifvc
\vcfalse
\ifvc
\input{vc}
\fi

\newif\ifcolor
\colortrue

\newcommand{\RR}{\mathbb R}
\newcommand{\ZZ}{\mathbb Z}
\newcommand{\QQ}{\mathbb Q}

\newcommand{\arr}[1]{\vv{#1}}

\newcommand{\bT}{\mathbf{T}}%

\newcommand{\ip}[1]{\bigl \langle #1 \bigr \rangle}

\newcommand{\ipm}[2]{\bigl \langle #1 \bigr \rangle_{#2}}

\newcommand{\sets}[1]{[\![#1]\!]}

\newcommand{\rfthmcite}[2]{\emph{\cite[#1]{#2}}}
\renewcommand{\tilde}{\widetilde}

\newcommand{\sA}{\mathscr{A}}

\newcommand{\sE}{\mathscr{E}}
\newcommand{\sF}{\mathscr{F}}
\newcommand{\sG}{\mathscr{G}}
\newcommand{\sK}{\mathscr{K}}

\newcommand{\sM}{\mathscr{M}}

\newcommand{\sP}{\mathscr{P}}

\newcommand{\sT}{\mathscr{T}}
\newcommand{\sW}{\mathscr{W}}

\newcommand{\cS}{\mathcal{S}}

\newcommand{\cX}{\mathcal{X}}%

\newcommand{\Sloop}{\cS_{\mathrm{l}}}
\newcommand{\Sparallel}{\cS_{\mathrm{p}}}
\newcommand{\Ssingle}{\cS_{\mathrm{s}}}
\newcommand{\Eloop}{E_{\mathrm{l}}}
\newcommand{\Eparallel}{E_{\mathrm{p}}}
\newcommand{\Esingle}{E_{\mathrm{s}}}
\newcommand{\Rloop}{R_{\mathrm{l}}}
\newcommand{\Rparallel}{R_{\mathrm{p}}}
\newcommand{\Rsingle}{R_{\mathrm{s}}}

\DeclareMathOperator{\Aut}{Aut}

\DeclareMathOperator{\In}{In}
\DeclareMathOperator{\Out}{Out}
\newcommand{\lk}[1]{\log_{\geq #1}}

\DeclareMathOperator{\Tr}{Tr}

\swapnumbers
\theoremstyle{plain}
\newtheorem{thm}[subsection]{Theorem}
\newtheorem{lem}[subsection]{Lemma}

\newtheorem{prop}[subsection]{Proposition}

\newtheorem{cor}[subsection]{Corollary}
\theoremstyle{definition}
\newtheorem{defn}[subsection]{Definition}
\newtheorem{remark}[subsection]{Remark}

\catcode`\@=11
\def\@secnumfont{\bfseries}
\catcode`\@12

\begin{document}

\title[Hypergraph matrix models and generating functions]{Hypergraph
matrix models and generating functions}

\author{Paul E. Gunnells}
\address{Department of Mathematics and Statistics\\University of
Massachusetts\\Amherst, MA 01003-9305}
\email{gunnells@math.umass.edu}

\renewcommand{\setminus}{\smallsetminus}

\date{April 22, 2022} 

\thanks{The author was partially supported by NSF grant 1501832.  We
thank Mario Defranco for many helpful conversations.  We also thank
Alejandro Morales and Eric Fusy for helpful comments.}



\begin{abstract}
Recently we introduced the \emph{hypergraph matrix model} (HMM), a
Hermitian matrix model generalizing the classical Gaussian Unitary
Ensemble (GUE).  In this model the Gaussians of the GUE, whose moments
count partitions of finite sets into pairs, are replaced by formal
measures whose moments count set partitions into parts of a fixed even
size $2m\geq 2$.  Just as the expectations of the trace polynomials
$\Tr X^{2r}$ in the GUE produce polynomials counting unicellular
orientable maps of different genera, in the HHM these expectations
give polynomials counting certain unicelled edge-ramified CW complexes
with extra data that we call \emph{(orientable CW) maps with
instructions}.  In this paper we describe generating functions for
maps with instructions of fixed genus and with the number of vertices
arbitrary.  Our results are motivated by work of Wright.  In
particular Wright computed generating functions of connected graphs of
fixed first Betti number as rational functions in the rooted tree
function $\sT (x)$, given as the solution to the functional relation
$x = \sT e^{-\sT}$.
\end{abstract}

\maketitle

\ifvc
\let\thefootnote\relax
\footnotetext{Base revision~\GITAbrHash, \GITAuthorDate,
\GITAuthorName.}
\fi

\section{Introduction}\label{s:intro}

\subsection{}\label{ss:defofgue} We begin by recalling the Gaussian
Unitary Ensemble (GUE) matrix model and its connection to counting
orientable maps.  For more information, we refer to Harer--Zagier
\cite{harer.zagier}, Etingof \cite[\S 4]{etingof}, Lando--Zvonkin
\cite{lz}, and Eynard \cite{eynard}.

Let $V$ be the $N^{2}$-dimensional real vector space of
$N\times N$ complex Hermitian matrices equipped with Lebesgue measure.
For any polynomial function $f\colon V \rightarrow \RR $, define
\begin{equation}\label{eq:expectation}
\ip{f} = \frac{\int_{V} f (X) \exp (-\Tr X^{2}/2)\,dX}{\int_{V} \exp (-\Tr X^{2}/2)\,dX},
\end{equation}
where $\Tr (X) = \sum_{i} X_{ii}$ is the sum of the diagonal entries.
Let $k\geq 0$ be an integer, and consider \eqref{eq:expectation}
evaluated on the function on $V$ given by taking the trace of the
$k$th power:
\begin{equation}\label{eq:basicint}
\ip{\Tr X^{k}}.
\end{equation}
We denote by $P_{k} (N)$ the quantity \eqref{eq:basicint} as a
function of $N$.  For $k$ odd \eqref{eq:basicint} clearly vanishes.
On the other hand, it turns out that if $k=2r$ is even, then
$P_{2r}(N)$ is an integral polynomial in $N$ of degree $r+1$ with the
following combinatorial interpretation.

Let $\Pi$ be a polygon with $2r$ sides.  We define a \emph{pairing}
$\pi $ of the sides of $\Pi $ to be a partition of the edges into
disjoint pairs.  We say that $\pi $ is an \emph{oriented pairing} if
orientations have been chosen for the edges, and is a \emph{positively
oriented pairing} if the edges have been oriented such that as one
moves clockwise around the polygon, paired edges appear with opposite
orientations.  Any positively oriented pairing $\pi $ of the sides of
$\Pi$ determines a \emph{rooted oriented unicellular map} $M_{\pi }$.
By definition this means $M_{\pi }$ is a closed orientable topological
surface with an embedded graph---namely the images of the edges and
vertices of $\Pi$---such that the complement of the graph is a
$2$-cell, and the graph has a distinguished vertex and a distinguished
edge containing that vertex.(\footnote{Throughout, we allow graphs to
have loops and multiple edges.})  Let $v(\pi )$ be the number of
vertices in this graph.  Then we have
\begin{equation}\label{eq:surfsum}
P_{2r}(N) = \sum_{\pi } N^{v(\pi )},
\end{equation}
where the sum is taken over all positively oriented pairings of the edges of
$\Pi$.  For example, we have
\begin{equation}\label{eq:hzpolys}
P_{4}(N) = 2N^{3}+N, \quad P_{6} (N) = 5N^{4}+10N^{2}, \quad P_{8} (N)
= 14N^{5}+70N^{3} + 21N. 
\end{equation}
The three positively oriented pairings of a square giving $P_{4} (N)$ are shown
in Figure \ref{fig:ccn4}.

\begin{figure}[htb]
\psfrag{n3}{$N^{3}$}
\psfrag{n}{$N^{1}$}
\begin{center}
\includegraphics[scale=0.3]{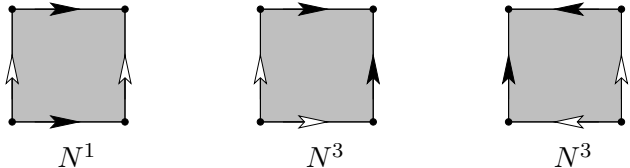}
\end{center}
\caption{Computing $P_{4} (N) = 2N^{3}+N$.  The leftmost surface is a
torus with embedded graph having one vertex.  The next two are the
$2$-sphere with embedded graph having three vertices.}\label{fig:ccn4}
\end{figure}

\subsection{} In \cite{ncontribution} we defined a generalization of
the GUE that we call the \emph{hypergraph matrix model} (HMM).  The
HMM depends on an integer $m\geq 1$; when $m=1$ the HMM coincides with
the GUE.  The connection with hypergraphs is that just as the
$1$-dimensional GUE can be used to construct generating functions of
graphs, the $1$-dimensional HMM can be used to construct generating
functions of certain hypergraphs.  For more details about this
connection, we refer to \cite[\S 3]{etingof} and \cite[Section
2]{ncontribution}.

To explain the HMM, we first recall the combinatorial interpretation
of the Gaussian moments.  Let $x$ be a real variable.  Suppose we
normalize the measure $dx$ on $\RR$ such that $\int_{-\infty}^{\infty}
e^{-x^{2}/2} \, dx = 1$. Consider the moment
\begin{equation}\label{eq:gaussianmoment}
\ip{x^{k}} :=
\int_{-\infty}^{\infty} x^{k} e^{-x^{2}/2}\, dx.
\end{equation}
If $k$ is odd, we have $\ip{x^{k}}=0$.  If $k$ is even, then it is
well known that $\ip{x^{k}}$ is the \emph{Wick number}
\[
W_{2} (k) = \frac{k!}{2^{k/2} (k/2)!},
\]
which counts the pairings of the finite set $\sets{k} := \{1,\dotsc ,k
\}$.  The Gaussians appear in the GUE when one writes the variable matrix
$X$ in terms of the underlying $N^{2}$ real coordinates.  Indeed the
measure $\exp (-\Tr X^{2})\,dX$ is then just a product of Gaussians, although with
different normalizations for the diagonal and off-diagonal
coordinates.  This is the underlying reason why the expectations of
the traces \eqref{eq:expectation} are related to pairings.

In the HMM we replace the product of Gaussians $\exp (-\Tr X^{2}/2)\,dX$ in
\eqref{eq:expectation} with a formal measure $dX_{m}$ that is a product of
certain one-variable measures $dx_{m}$.  The moment
\begin{equation}\label{eq:hmmmoments}
\ipm{x^{k}}{m} := \int x^{k}\, dx_{m}
\end{equation}
is defined to be $0$ if $m$ does not divide $k$, and otherwise is
given by the \emph{generalized Wick number} 
\begin{equation}\label{eq:genwicknumber}
W_{2m} (k) = \frac{k!}{m!^{k/m} (k/m)!},
\end{equation}
which counts the number partitions of $\sets{k}$ into disjoint subsets
of order $2m$.  Again using the real coordinates to define the product
measure $dX_{m}$, one can then define the expectations
\begin{equation}\label{eq:introbasicint}
\ipm{\Tr X^{k}}{m} := \int_{V} \Tr X^{k}\, dX_{m}.
\end{equation}
We let $P^{(m)}_{k} (N)$ denote \eqref{eq:introbasicint} as a function
of $N$.

The main results of \cite{ncontribution} show that
\eqref{eq:introbasicint} vanishes unless $k=2mr$, and otherwise
$P^{(m)}_{2mr} (N)$ is an integral polynomial of degree $r+1$.  For instance, when $m=2$ we have
\begin{multline*}
P^{(2)}_{8} (N) = 6N^3 + 21N^2 + 8N, \quad P^{(2)}_{12}(N) = 57N^4 + 715N^3
+ 2991N^2 + 2012N, \\
P^{(2)}_{16} (N) = 678N^5 + 19405N^4 + 228190N^3 + 1151300N^2 + 1228052N.
\end{multline*}
In \cite{ncontribution} several combinatorial interpretations are
given for the polynomial $P^{(m)}_{2mr} (N)$, including (i) counting
walks on digraphs, (ii) counting points on subspace arrangements over
finite fields and (iii) enumerating \emph{unicellular CW maps with
instructions}, a generalization of maps that have ramified edges and
extra gluing data.

\subsection{} The goal of the current paper is to give generating
functions for the coefficients of $P_{2mr}^{(m)} (N)$.  The case $m=1$
(GUE), which corresponds to counting unicellular maps, has been
considered by a variety of authors.  For a (rather incomplete) list,
one can refer to Harer--Zagier \cite{harer.zagier}, Eynard
\cite{eynard}, Lass \cite{lass}, Chapuy \cite{chapuy1,chapuy2},
Chapuy--F\'eray--Fusy \cite{cff}, Bernardi \cite{bernardi}, and
Goulden--Nica \cite{goulden.nica}.

More precisely, we consider the rational polynomials
\begin{equation}\label{eq:weighting}
\sP^{(m)}_{2mr} := \frac{1}{2mr} P^{(m)}_{2mr} (N) \in \QQ [N],
\end{equation}
whose coefficients count \emph{weighted unicellular CW maps with
instructions}: the contibution of each map has been multiplied by the
inverse of the order of its automorphism group.  For our purposes it
is more convenient to work with the coefficients in the falling
factorial basis.  Thus let $(N)_{k} = N (N-1)\dotsb (N-k+1)$, and
write \[ \sP^{(m)}_{2mr} (N) = \sum_{l=0}^{r+1} \alpha^{(m)}_{l,r}
(N)_{r+1-l},
\]
For $g\geq 0$, we put
\begin{equation}\label{eq:defofG}
\sG_{g}^{(m)}(x) = \sum_{v\geq 1} \alpha^{(m)}_{g,g-1+v} x^{v} \in \QQ \sets{x}.
\end{equation}
Thus $\sG^{(m)}_{0}$ is the generating function of the leading
coefficients of the $\sP^{(m)}$ in the falling factorial basis, the
series $\sG^{(m)}_{1}$ is that of the subleading coefficients, and so
on. 
For example, when $m=1$ we have
\begin{align*}
\sG_{0}^{(1)} &= x^{2}/2 + x^{3}/2 + 5x^{4}/6 + 7x^{5}/4 + 21x^{6}/5 + \dotsb,  \\ 
\sG_{1}^{(1)} &= x/2 + 3x^{2}/2 + 5x^{3} + 35x^{4}/2 + 63x^{5} + 231x^{6} + \dotsb, \\
\sG_{2}^{(1)} &= 3x/4 + 15x^2/2 + 105x^3/2 + 315x^4 + 3465x^5/2 + \dotsb,
\end{align*}
and when $m=2$,
\begin{align*}
\sG_{0}^{(2)} &=  x^2/4 + 3x^3/4 + 19x^4/4 + 339x^5/8  + 927x^{6}/2 +\dotsb ,\\
\sG_{1}^{(2)} &=  x/4 + 39x^2/8 + 1057x^3/12 + 26185x^4/16 + 157754x^{5}/5 + \dotsb ,\\
\sG_{2}^{(2)} &=  35x/8 + 1845x^{2}/4 + 180785x^{3}/8 + 862005x^4 + \dotsb .
\end{align*}
We call $g$ the genus, since when $m=1$ it agrees with the genus of
the surface underlying the map.  The series $\sG^{(m)}_{g}$ are our
basic objects of study.  More examples can be found in Appendix
\ref{app}.

We remark that the series $\sG^{(m)}_{0}$ were already essentially
computed in \cite{gencat}.  Indeed, we have
\[
\sG_{0}^{(m)} (x) = \sum_{v\geq 1} \frac{C_{v}^{(m)}}{2mv}x^{v+1},
\]
where $C_{v}^{(m)}$ is a \emph{hypergraph Catalan number}
\cite[Definition 2.5]{gencat}.

\subsection{} We are almost ready to state our results, but before we
do we must recall work of Wright \cite{wright1}, whose results
inspired ours.  Let $V (G)$ (respectively, $E (G)$) be the vertices
(resp., edges) of $G$, and let $v (G) = |V (G)|$ (resp., $e (G) = |E
(G)|$). The \emph{excess} of $G$ is defined to be $e (G)-v (G)$. Wright
studied generating functions of connected graphs of fixed
excess.  This work was revisited by Janson--Knuth--\L
uczak--Pittel \cite{jklp}, at least for excesses $\leq 1$; other
expositions appear in Flajolet--Sedgewick \cite[II.5]{fs} and
Dubrovin--Yang--Zagier \cite[\S 3.1]{dyz}.(\footnote{The authors of
\cite{dyz} do not explicitly refer to Wright's work, but they do prove
the relevant results we need.})  In the study of maps Wright's method
is called the method of \emph{scheme decompositions} and is due to
Chapuy--Marcus--Schaeffer \cite{cms}; for more about
this see \S\ref{ss:schemes}.

For $g\geq 0$ let $\Gamma(g)$ be the set of connected graphs $G$ with
first Betti number $g$.  The quantity $g$ is called the \emph{loop
number} in the physics literature; it exceeds the excess by 1: $ g = e
(G) - v (G) + 1$.  Let $\Aut G$ be the full automorphism group of $G$
generated by permuting vertices and parallel edges, as well as
flipping and permuting loops.  We consider the generating function
\begin{equation}\label{eq:Wg1}
\sW_{g} (x) = \sum_{v\geq 1} \sum_{G\in \Gamma (g)} \frac{x^{v (G)} }{|\Aut G|},
\end{equation}
which we regard as enumerating \emph{weighted} graphs: the coefficient
$[x^{v}]\sW_{g}$ is not the number of connected graphs with first
Betti number $g$, but rather the sum of them weighted by the inverses
of the orders of their automorphism groups.  For example, we have
\begin{align*}
\sW_{0} &= x + x^2/2 + x^3/3 + 2x^4/3 + 25x^5/24 + 9x^6/5 + 2401x^7/720 + \dotsb, \\
\sW_{1} &= x/2 + 3x^2/4 + 17x^3/12 + 71x^4/24 + 523x^5/80 + 899x^6/60 + \dotsb, \\
\sW_{2} &= x/8 + 7x^2/12 + 101x^3/48 + 83x^4/12 + 12487x^5/576 + 3961x^6/60 + \dotsb .
\end{align*}
Wright showed that for $g\geq 1$ the series \eqref{eq:Wg1} can be
computed in terms of what is essentially the series for the $g=0$
case.  Let $\sT$ be the inverse series of $xe^{-x}$:
\begin{equation}\label{eq:Tintro}
\sT (x)  = x+x^{2}+3x^{3}/2 + 8x^{4}/3 + \dotsb .
\end{equation}
The series $\sT$ is the generating function of weighted rooted
trees.(\footnote{The series $\sT $ is sometimes called the
\emph{Lambert series} because it is essentially the power series for
Lambert's $W$-function \cite[\S4.13]{NIST:DLMF}.  Apparently the
series $\sT$ was first considered by Eisenstein \cite{eisenstein},
although he did not interpret the coefficients in terms of trees.})
Let $\Gamma_{\geq 3} (g) \subset \Gamma (g)$ be the finite subset of
graphs with all vertices having degree at least $3$.  Then if $g\geq
2$, we have
\begin{equation}\label{eq:Wgintro}
\sW_{g} (x) = \sum_{G\in \Gamma_{\geq 3} (g)} \frac{1}{|\Aut G|}\frac{\sT^{v
(G)}}{(1-\sT)^{e (G)}}, \quad g\geq 2.
\end{equation}
One can also give expressions for $g=0,1$ in terms of $\sT$, but
these are exceptional (\S \ref{s:wright}).

\subsection{} Our main results, Theorems \ref{thm:G0}, \ref{thm:G1},
and \ref{thm:main}, give expressions for the series $\sG^{(m)}_{g}$
respectively for $g=0$, $g=1$, and $g\geq 2$.  Like Wright's theorem,
the expression for $g\geq 2$ is a sum over the finite set
$\Gamma_{\geq 3} (g)$ and is given in terms of certain tree functions,
but there are other ingredients involved.  One must also sum over
balanced digraph structures on ``thickenings'' of graphs in
$\Gamma_{\geq 3} (g)$, and over sets of oriented spanning trees of
directed graphs.

We remark that although our results were inspired by those of
\cite{wright1}, our results do \emph{not} generalize Wright's: the
series $\sG_{g}^{(m)}$ is not a generalization of $\sW_{g}$, and in
particular $\sG^{(1)}_{g} \not = \sW_{g}$.

\subsection{} Here is a guide to the paper.  In \S\ref{s:coeffs} we
recall results from \cite{ncontribution} that explain how to compute
the polynomials $P^{(m)}_{2mr} (N)$ in terms of walks on graphs.  In
\S\ref{s:wright} we present Wright's computation of his generating
functions.  Section \ref{s:genus0} gives the computation of the series
$\sG_{0}^{(m)}$ (Theorem \ref{thm:G0}) and the additional tree
functions needed to compute $\sG^{(m)}_{g}$ for $g\geq 1$; these
series might be of independent interest.  In \S\ref{s:genus1} we
compute $\sG_{1}^{(m)}$ (Theorem \ref{thm:G1}).  The next three
sections \S\S \ref{s:genusgsingle}--\ref{s:genusgloops} give the
ingredients we need to compute $\sG_{g}^{(m)}$ for $g\geq 2$.  In
\S\ref{s:main} we state and prove our main theorem.  We conclude with
some examples.  Section \ref{s:example} discusses some examples and
complements, and in Appendix \ref{app} we give examples of
$\sG_{g}^{(m)}$ for small $m$ and $g$.

\section{Coefficients of $P_{2mr}^{(m)} (N)$ and walks on
graphs}\label{s:coeffs}

\subsection{} In this section we recall results from
\cite{ncontribution}.  The main result is that $P^{(m)}_{2mr}$ can be
computed in the falling factorial basis as a finite sum over certain
graphs $G$ involving three quantities:

\begin{enumerate}
\item A \emph{walks} contribution, which counts Eulerian
tours on digraphs built from $G$.
\item A \emph{moment} contribution, which incorporates the
moments of the measure $dX_{m}$.
\item A contribution from the automorphism group of $G$.
\end{enumerate}

We treat each of these in turn, beginning with the walks contribution.

\subsection{} Let $\Gamma [r]$ be the set of connected graphs with $r$
edges and any number of vertices.  Any $G\in \Gamma [r]$ has a
canonical \emph{thickening} to a graph $\widetilde{G}$ with $2mr$ edges: we
simply replace each edge $e$ in $G$ with $2m$ edges running between
the endpoints of $e$.  Note that $\widetilde{G}$ depends on the
parameter $m$, although we omit it from the notation.

\subsection{} Let $D (G)$ be the set of balanced digraph structures on
the thickening $\widetilde{G}$.  Thus an element of $D (G)$ is an
assignment of orientations to the edges of $\widetilde{G}$ such that
at each vertex $v$ the total indegree $\In (v)$ equals the total
outdegree $\Out (v)$.  The set $D (G)$ is clearly finite, and since
the degree of each vertex of $\widetilde{G}$ is even, the set $D (G)$
is nonempty.

\subsection{}The set $D (G)$ is naturally in bijection with the set of
lattice points in a polytope.  To see this, first we define a
partition of the edges $E (G)$. Define an equivalence relation on
edges by $e\sim e'$ if (i) $e$ and $e'$ are loops at a common vertex
or (ii) $e$ and $e'$ are parallel edges running between the same two
distinct vertices.  Let $\cS (G)$ be the set of equivalence classes.
We call the classes of type (i) the \emph{loop classes} and those of
type (ii) the \emph{nonloop  classes}.  We write
\[
\cS (G) = \Ssingle  (G) \sqcup \Sparallel (G) \sqcup \Sloop (G).
\]
where
\begin{itemize}
\item $\Ssingle (G)$ is the nonloop classes represented by singletons,
\item $\Sparallel (G)$ is the nonloop classes not represented by
singletons, and
\item $\Sloop (G)$ is the loop classes,
\end{itemize}
Thus for example $S\in \Sparallel (G)$ is the equivalence class of a full
set of $>1$ parallel edges between two distinct vertices in $G$.  We
also have the corresponding partition of the edges
\[
E (G) = \Esingle  (G) \sqcup \Eparallel  (G) \sqcup \Eloop  (G).
\]

For $e\in \Esingle \cup \Eparallel$, let $\arr{e}$ be $e$ with a fixed
arbitrary orientation.  We fix these orientations so that parallel
edges are oriented in the same direction.  For each $S\in \Ssingle
\cup \Sparallel $, choose two real variables $a_{S},b_{S}$.  We take
$a_{S}$ (respectively $b_{S}$) to count the edges in a digraph
structure on $\tilde{G}$ running parallel to $e\in S$ and in the same
(resp., opposite) direction of $\arr{e}$.  For the digraph to be
balanced, the $a_{S}, b_{S}$ must satisfy
\begin{align}
a_{S}, b_{S} \geq 0, \quad a_{S} + b_{S} = 2m|S|  \quad \text{for all $S\in \Ssingle \cup \Sparallel$},\label{eq:polytope1}\\
\sum_{S} \varepsilon (S,v) (a_{S} - b_{S}) = 0 \quad \text{for all $v\in V (G)$}.\label{eq:polytope2}
\end{align}
In \eqref{eq:polytope2} the sum is taken over all classes $S$ containing
an edge $e$ meeting $v$, and $\varepsilon (S,v)=1$ if $v$ is the tail of
$\arr{e}$ and equals $-1$ otherwise.

The real solutions to \eqref{eq:polytope1}--\eqref{eq:polytope2}
define a lattice polytope(\footnote{This polytope is a variant of the
flow polytope.}) $X_{m} (G)$, and it is clear that $D (G)$ is in
bijection with the lattice points in $X_{m}
(G) $.  Given an integral solution, we call the integers
\[
\bigl\{ a_{S}, b_{S} \bigm| S \in \Ssingle \cup \Sparallel \bigr\}
\]
the \emph{digraph parameters}.  Note that the solution
$a_{S} = b_{S} = m|S|$ for all $S$ gives a balanced digraph for any
$\tilde{G}$; we call this the \emph{canonical digraph structure}.

It is clear that any $\gamma \in D (G)$ has an Eulerian tour.
Following Lass \cite{lass}, we call two Eulerian tours
\emph{essentially different} if one cannot be obtained from the other
by permuting loops or parallel edges.  We often abbreviate essentially
different Eulerian tour by \emph{EDET}.

\begin{defn}\label{def:wgamma}
Let $G$ be a graph and let $\gamma \in D (G)$ be a balanced digraph
structure on the thickening $\widetilde{G}$.  We define the
\emph{walks contribution} $W (\gamma)$ to be the number of essentially
different Eulerian tours on $\gamma$.
\end{defn}

\subsection{} Next we consider the moment contribution $M (\gamma)$.
Let $x,y$ be real variables and let $z = x + y\sqrt{-1}$ be a complex
variable.  We use the generalized Wick numbers
\eqref{eq:genwicknumber} to define moments on powers of $x,y$ and by
linearity on powers of $z, \bar z$.  For our purposes we will need two
different normalizations for these moments:
\begin{itemize}
\item \textbf{Loop normalization:} We put $\ipm{x^{k}}{m} = \ipm{y^{k}}{m} = W_{m}
(k)$ if $m \mid k$, and $0$ otherwise.
\item \textbf{Nonloop normalization:} We put $\ipm{x^{k}}{m} = \ipm{y^{k}}{m} = W_{m}
(k)/2^{k}$ if $m \mid k$, and $0$ otherwise.
\end{itemize}

\subsection{}\label{ss:subsets} Now let $\gamma \in D (G)$ be an
balanced digraph structure on $G$ with parameters $\{ a_{S}, b_{S}\}$.
We define $M (\gamma)$ by taking a product over the edges of certain
Wick numbers.

\begin{itemize}
\item \textbf{Loop subsets:} Let $S\in \Sloop (G)$ be a loop subset.  We define
\[
M(S) = \ipm{x^{2m|S|}}{m},
\]
where the moment is computed using the loop normalization.
\item \textbf{Nonloop subsets:} Let $S \in \Ssingle (G) \cup
\Sparallel (G)$ be a nonloop subset.  Let $a_{S}, b_{S}$  be the
corresponding digraph parameters.  We put
\[
M (S) = \ipm{z^{a_S}\bar z^{b_S}}{m},
\]
where the moment is computed using the nonloop normalization.
\end{itemize}

\begin{defn}\label{def:egamma}
We define the \emph{moment contribution} $M (\gamma)$ to be the
product 
\[
\prod_{S\in \cS (G)} M (S).
\]
\end{defn}

\subsection{} Finally we define the contribution from the automorphism
group of $G$.  Let $\Aut G$ be the full group of automorphisms of $G$
generated by permuting vertices and edges.  Since $G$ has loops and
parallel edges, $\Aut G$ also includes automorphisms that fix the
vertices, namely permuting parallel edges, permuting loops at a
vertex, and interchanging the half-edges of a loop at a vertex.  Let
$\Aut^{v} G \subset \Aut G$ be the subgroup that acts trivially on the
vertices.

\begin{defn}\label{def:autv}
We define $\Aut_{v}G$ to be the quotient  $\Aut G / \Aut^{v} G$.
\end{defn}

\subsection{} We can now state the result from \cite{ncontribution}
that gives an explicit expression for the coefficients of the
polynomials $P^{(m)}_{2mr} (N)$.  Recall that $(N)_{k}$ denotes the
falling factorial $N (N-1)\dotsb (N-k+1)$.

\begin{thm}\label{s:ncontrib}
\rfthmcite{Proof of Theorem 3.6}{ncontribution} Let $\Gamma [r]$ be
the set of connected graphs with $r$ edges and with any number of
vertices.  Let $D (G)$ be the set of balanced digraph structures on
the thickening $\widetilde{G}$.  Let $W (\gamma)$, $M (\gamma)$, and
$\Aut_{v}G$ be defined as in Definitions \ref{def:wgamma},
\ref{def:egamma}, and \ref{def:autv}.  Then we have
\[
P^{(m)}_{2mr}(N) = \sum_{G\in \Gamma [r]} \frac{(N)_{v
(G)}}{|\Aut_{v}G|} \sum_{\gamma \in D (G)} W (\gamma)M (\gamma).
\]
\end{thm}

For an example of computing $P^{(m)}_{2mr} (N)$ using this result, we
refer to \cite[Example 3.7]{ncontribution}. 

\begin{cor}\label{cor:Gg}
For $g\geq 1$ we have 
\begin{equation}\label{eq:Gexp}
\sG_{g}^{(m)}(x) = \frac{1}{2m}\sum_{G\in
\Gamma (g)} \frac{x^{v (G)}}{(v (G) +g-1)|\Aut_{v} G|} \sum_{\gamma \in D (G)} W (\gamma)M (\gamma).
\end{equation}
If $g=0$, the series $\sG^{(m)}_{0}$ is given by the same expression \eqref{eq:Gexp},
except that the sum is taken over graphs with at least $2$ vertices.
\end{cor}

\begin{remark}\label{rem:mod4}
The contribution $M (\gamma)$ can vanish for certain $\gamma \in D
(G)$.  Indeed, according to \cite[Lemma 3.4]{ncontribution}, the
moment $\ipm{z^{a}\bar z^{b}}{m}$ is nonzero only if the exponents
satisfy $4 \mid a-b$.  Thus there will typically
be many lattice points in $X_{m} (G)$ that give balanced digraphs but
will not contribute to \eqref{eq:Gexp}.  If one wishes to keep only
the nonzero terms in \eqref{eq:Gexp}, one can replace the canonical
lattice in the space defining $X_{m} (G)$ with a sublattice.
\end{remark}

\subsection{}
We conclude this section by recalling the BEST theorem for counting
Eulerian tours on balanced digraphs.  For more information we refer to
\cite[5.6]{ec2}.

Let $\gamma$ be a balanced digraph with vertex set $V (\gamma )$.  For
any vertex $v\in V (\gamma )$, an \emph{oriented spanning tree} $T$ rooted at
$v$ is a spanning tree of the undirected graph underlying $\gamma $
such that the edges in $\gamma$ point towards $v$.  Let $\sigma(\gamma
, v)$ be the number of oriented spanning trees of $\gamma$ with root
$v$.

\begin{thm}\label{thm:best}
\rfthmcite{BEST Theorem 5.6.2}{ec2} Let $\gamma$ be a balanced digraph.
Let $e$ be an edge of $\gamma$, and let $v$ be the initial vertex of
$e$.  Then the number of Eulerian tours of $\gamma$ with initial
directed edge $e$ is given by
\begin{equation}\label{eg:best}
\sigma (\gamma , v) \prod_{u\in V (\gamma)} (\Out (u) - 1)!.
\end{equation}

\end{thm}
\section{Wright's generating functions of graphs}\label{s:wright}

\subsection{}\label{ss:g0} The goal of this section is to recall the
proof of \eqref{eq:Wgintro} and the related expressions for $g=0,1$ in
terms of the tree function $\sT (x)$. The original ideas go back to
Wright \cite{wright1}; our reference is \cite[\S 3.1]{dyz}.  While
these results are not explicitly needed in the sequel, the computation
of $\sG_{g}^{(m)}$ is analogous, and thus they serve as a warmup for
later results.

\subsection{}
First we consider $g=0$.  Writing 
\[
\sW_{0} (x) = \sum_{v\geq 1} W_{0,v} x^{v},
\]
we see that $W_{0,v}$ counts the weighted trees with $v$ vertices.
This is almost the same as $\sT$ \eqref{eq:Tintro}; the only
difference between $[x^{v}]\sW_{0}$ and $[x^{v}]\sT$ is that the
latter counts weighted rooted trees with $v$ vertices.  There are $v$
ways to choose a root in a tree contributing to $W_{0,v}$.  This
implies $[x^{v}]\sT = v W_{0,v}$, and thus
\[
\sT (x) = x\frac{d}{dx}\sW_{0} (x).
\]
From this we have
\begin{equation}\label{eq:W0}
\sW_{0} (x) = \int x^{-1}\sT (x)\,dx,
\end{equation}
which gives an expression for $\sW_{0}$ in terms of the tree function
$\sT$. 


\subsection{}\label{ss:g1} Next we consider $g=1$.  A connected graph
$G$ with $g=1$ consists of a unique $k$-cycle $C_{k}$ with no
backtracking and with rooted trees attached to its vertices (Figure
\ref{fig:cycle}).  We say $G$ is built from $C_{k}$ by \emph{grafting
rooted trees at the vertices}.  The contribution to $\sW_{1}$ of all
such graftings onto $C_{k}$ is $\sT (x)^{k}/2k$.  The $\sT^{k}$ comes
from enumerating the possible choices of $k$ ordered weighted rooted
trees, and the $2k$ is the order of the automorphism group of $C_{k}$.
Thus
\begin{equation}\label{eq:W1}
\sW_{1} (x) = \frac{1}{2}\sum_{k\geq 1} \frac{\sT^{k}}{k} =
\frac{1}{2}\log \Bigl(\frac{1}{1-\sT}\Bigr).
\end{equation}

\begin{figure}[htb]
\begin{center}
\includegraphics[scale=0.15]{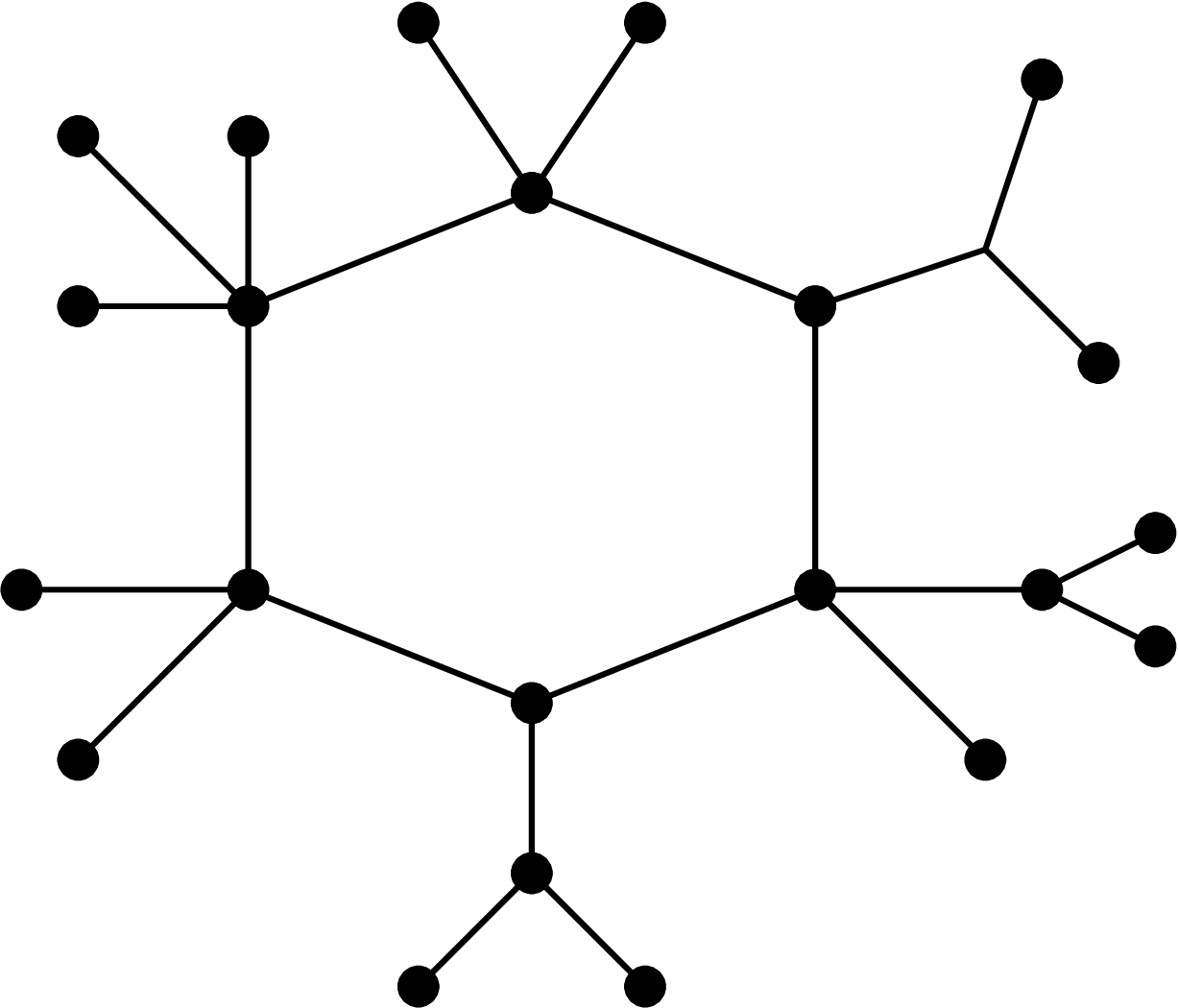}
\end{center}
\caption{Grafting rooted trees onto a cycle.\label{fig:cycle}}
\end{figure}

\subsection{}\label{ss:g2} Finally we consider $g\geq 2$.  Let $H\in
\Gamma (g)$.  Then there is a unique graph $G \in \Gamma_{\geq 3}
(g)$ that can be constructed from $H$.  We first erase any degree $1$
vertices and their attached edges.  We continue doing this until there
are no more vertices of degree $1$ to obtain a graph $G'$.  Next if
there are any degree $2$ vertices, we erase them and join the
corresponding edges into a single edge.  After eliminating all degree
$2$ vertices, the result is the graph $G$.  We call $G$ the
\emph{reduction} of $H$ (Figure \ref{fig:reduction}).

\begin{figure}[htb]
\psfrag{ar}{$\longrightarrow $}
\psfrag{H}{$H$}
\psfrag{Gp}{$G'$}
\psfrag{G}{$G$}
\begin{center}
\includegraphics[scale=0.15]{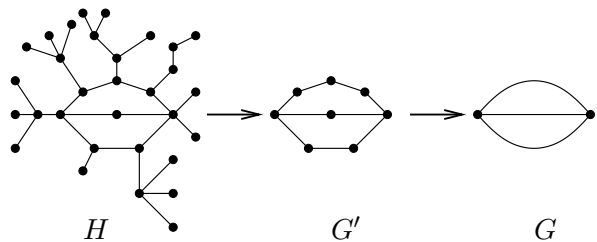}
\end{center}
\caption{Reducing a graph $H\in \Gamma (2)$ to $G\in \Gamma_{\geq 3} (2)$.\label{fig:reduction}}
\end{figure}

It is clear that the reduction map $\Gamma (g)\rightarrow \Gamma_{\geq 3}
(g)$ is surjective.  Conversely, any graph in $\Gamma (g)$ can be
built by starting with a graph in $\Gamma_{\geq 3} (g)$, subdividing
edges by adding some degree $2$ vertices, then grafting rooted trees
onto the original vertices and those of degree $2$.  The contribution
of $G\in \Gamma_{\geq 3} (g)$ to $\sW_{g}$ is thus
\[
\frac{1}{|\Aut G|}\sT^{v (G)} (1+\sT +\sT^{2} + \dotsb)^{e (G)} = \frac{1}{|\Aut G|}\frac{\sT^{v (G)}}{(1-\sT)^{e (G)}}.
\]
Summing over all $G\in \Gamma_{\geq 3} (g)$ we obtain \eqref{eq:Wgintro}:
\begin{equation}\label{eq:Wg}
\sW_{g} (x) = \sum_{G\in \Gamma_{\geq 3} (g)} \frac{1}{|\Aut G|}\frac{\sT^{v
(G)}}{(1-\sT)^{e (G)}}, \quad g\geq 2.
\end{equation}

\subsection{} To finish the discussion, we show that the indexing set
$\Gamma_{\geq 3} (g)$ in \eqref{eq:Wg} is finite.  Indeed, if $G\in
\Gamma_{\geq 3} (g)$ has $d_{k}$ vertices of degree $k$, then
\[
v (G) = \sum_{k\geq 3} d_{k}, \quad e (G) = \frac{1}{2}\sum_{k\geq 3} kd_{k},
\]
which implies 
\begin{equation}\label{eq:degs}
2g-2=\sum_{k\geq 3} (k-2)d_{k}.
\end{equation}
This means the degree sequence $(d_{3},d_{4},d_{5},\dotsc )$
corresponds uniquely to a partition of $2g-2$, and so there are only
finitely many such sequences for each $g$.  Moreover each degree
sequence leads to only finitely many possibilities for a graph, so the
sum \eqref{eq:Wg} is finite.  To summarize, 
\begin{itemize}
\item all the $\sW_{g}$, $g\geq 2$, are rational functions in $\sT$;
\item the
series $\sW_{1}$ is the log of a rational function in $\sT$; and  
\item the
series $\sW_{0}$ is the integral of $x^{-1}\sT$.
\end{itemize}

\begin{remark}\label{rmk:genfuns}
We remark it is easy to see that each solution of \eqref{eq:degs} does lead to at least
one graph in $\Gamma_{\geq 3} (g)$, and thus 
\[
|\Gamma_{\geq 3} (g)| \geq p (2g-2),
\]
where $p (n)$ is the number of partitions of $n$.  Wright
\cite[p.~328]{wright1} indicates that $|\Gamma_{\geq 3} (g)|$ equals
$3, 15, 107$ respectively for $g=2, 3, 4$, and that for $g=5$ one
expects between $900$ and $1000$ graphs.  The rapid growth of
$|\Gamma_{\geq 3} (g)|$ means \eqref{eq:Wg} (as written) is only
practical for small values of $g$.  On the other hand, using
generating functions one can compute the weighted sum
\begin{equation}\label{eq:autsum}
\sum_{G\in \Gamma_{\geq 3} (g)} \frac{x^{v (G)}y^{e (G)}}{|\Aut G|}
\end{equation}
without enumerating $\Gamma_{\geq 3} (g)$.(\footnote{One method is to
modify the computations in \cite[Theorem 2.3]{ncontribution}.  First
one can set $\xi_{1}=\xi_{2}=0$ to restrict the
graphs with degrees $\geq 3$.  One can also include another marking
variable to record how many edges one obtains.  Then after taking the
logarithm one recovers the connected graphs, and can easily extract
the expression \eqref{eq:autsum} for any $g$.}) This is
sufficient for computing $\sW_{g}$ and works as long as $p (2g-2)$ is
not too big.  For example for $g=5$ the sum \eqref{eq:autsum} equals
\begin{multline*}
565x^{8}y^{12}/128 + 12455x^{7}y^{11}/768 + 26581x^{6}y^{10}/1152 + 12227x^{5}y^{9}/768 \\
+ 2089x^{4}y^{8}/384 + 9583x^{3}y^{7}/11520 + 27x^{2}y^{6}/640 + xy^{5}/3840,
\end{multline*}
and we have 
\begin{multline*}
\sW_{5} (x) = x/3840 + 7x^{2}/160 + 27101x^{3}/23040 + 29951x^{4}/1920 + 13112269x^{5}/92160 \\
+ 4940551x^{6}/4800 + 17587526771x^{7}/2764800 + 21216093173x^{8}/604800  + \dotsb 
\end{multline*}
\end{remark}

\section{The series $\sG_{0}^{(m)}$ and the higher tree functions}\label{s:genus0}

\subsection{}

Write $\sG_{0}^{(m)} (x) = \sum_{v\geq 1}\alpha_{0,v}^{(m)}x^{v+1}$.  By
\eqref{eq:Gexp} we have
\[
\alpha_{0,v}^{(m)} = \frac{1}{2mv}\sum_{\substack{G\in \Gamma (0)\\
v (G) = v+1}} \frac{1}{|\Aut G|}\sum_{\gamma \in D (G)} W (\gamma)M (\gamma ).
\]
The condition on the sum implies $G$ is a tree $T$ with $v+1$
vertices.  The only balanced digraph structure on a thickened tree
$\tilde{T}$ is the canonical one $\gamma_{T}$, which we claim implies
$M (\gamma_{T})=1$.  Indeed, all edge subsets $S$ are singletons and
the digraph parameters $a_{S}, b_{S}$ all equal $m$. Thus $ M (S) =
\ipm{(z\bar z)^{m}}{m} = 1$, which implies $M (\gamma_{T}) = 1$.

Thus we have 
\begin{equation}\label{eq:hcndef}
2mv \alpha_{0,v}^{(m)} = \sum_{\substack{T\in \Gamma (0)\\
v (T) = v+1}} \frac{W (\gamma_{T})}{|\Aut T|}.
\end{equation}
The right hand side of \eqref{eq:hcndef} is by definition the
\emph{hypergraph Catalan number} $C_{v+1}^{(m)}$ \cite[Definition
2.5]{gencat}.  We write $\sF^{(m)} (x)$ for their ordinary generating function:
\begin{equation}\label{}
\sF^{(m)} (x) = \sum_{v\geq 0} C_{v}^{(m)}x^{v+1}.
\end{equation}
We will recall a combinatorial interpretation of these numbers below,
and also how to compute $\sF^{(m)}$.  At the moment we want to record
the exact relationship between $\sG^{(m)}_{0}$ and $\sF^{(m)}$, which
is essentially integration, just like \eqref{eq:W0}.  The final form
is slightly different from \eqref{eq:W0}, because the hypergraph
Catalan number $C^{(m)}_{0} = [x]\sF^{(m)}(x) = 1$ doesn't appear in
$\sG^{(m)}_{0} (x)$: the trivial reason that there is no polygon with
$0$ sides.

\begin{thm}\label{thm:G0}
We have 
\[
\sG^{(m)}_{0} (x) = \int x^{-2}\bigl(\sF^{(m)} (x) - x\bigr)\, dx.
\]
\end{thm}

\subsection{} Now we recall a combinatorial interpretation of the
hypergraph Catalans.  In fact we will do more.  The computation of
$\sG_{g}^{(m)}$ for $g\geq 1$ requires a generalization of the
hypergraph Catalans to numbers $C_{v,s}^{(m)}$ depending on an
additional integral parameter $s\geq 1$.  When $s=1$, we have
$C_{v,1}^{(m)} = C_{v}^{(m)}$.

The numbers $C_{v}^{(m)}$ count certain structures on plane trees.
Let $\cX$ be the set of plane trees on $mv+1$ vertices.  Given $X\in
\cX$, let $X_{i}, i\geq 1$ be the set of vertices at distance $i$ from
the root.  Let $L$ be a set of labels of order $v$.  We say $X$ is
\emph{admissibly $m$-labeled} if the following hold:
\begin{enumerate}
\item Each $X_{i}, i\geq 1$ is equipped with a disjoint decomposition into
subsets of order $m$ (in particular $m $ divides $ |X_{i}|$ for each $i$).
\item Each subset is assigned a distinct label from $L$.  We write $l
(x)$ for the label of $x\in X_{i}$.
\item If two vertices satisfy $l (x) = l (x')$, then the labels of
their parents agree.
\end{enumerate}
We consider admissible $m$-labelings of $\cX$ to be equivalent if one
is obtained from the other by permuting labels.  Figure
\ref{fig:2-labeling} shows an example of an admissibly $2$-labeled
plane tree. We denote by $\sA^{(m)}$ the set of all admissibly
$m$-labeled plane trees.

\begin{figure}[htb]
\psfrag{a}{$a$}
\psfrag{b}{$b$}
\psfrag{c}{$c$}
\psfrag{d}{$d$}
\psfrag{e}{$e$}
\psfrag{f}{$f$}
\begin{center}
\includegraphics[scale=0.20]{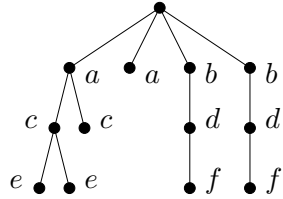}
\end{center}
\caption{An admissible $2$-labeling of a plane tree.\label{fig:2-labeling}}
\end{figure}

\begin{thm}\label{thm:admlab}
\rfthmcite{Theorem 4.9}{gencat} The hypergraph Catalan number
$C_{v}^{(m)}$ equals $|\sA^{(m)}|$, the number of admissibly
$m$-labeled plane trees with $mv+1$ vertices.
\end{thm}

\subsection{}
Now we define the numbers $C_{v,s}^{(m)}$ by modifying this
interpretation.  Let $s\geq 1$ be an integer.  Given a plane tree $X$,
let $x_{1},\dotsc ,x_{d}$ be the children of the root, written in
order.  We say $X$ has \emph{$d$ subtrees}.  Let $\xi$ be a
composition of $d$ into $s$ parts.  Then $\xi$ determines an ordered
decomposition of $\{x_{1},\dotsc ,x_{d} \}$ into $s$ contiguous
subsets, some of which might be empty.  We say the pair $(X,\xi)$ is a
plane tree with its subtrees \emph{placed into $s$ regions}.  A
example is shown in Figure \ref{fig:regions}.  The plane tree $X$ has
$4$ subtrees and $s=6$.  (We have omitted the vertex labels to keep
the figure cleaner.) In this example $\xi = (0,0,2,1,1,0)$, and thus
the $4$ subtrees of $X$ are being placed into $4$ regions.  We have
drawn bounding segments in the plane to indicate the $6$ regions
$R_{1}, \dotsc , R_{6}$.

\begin{figure}[htb]
\psfrag{r1}{$R_1$}
\psfrag{r2}{$R_2$}
\psfrag{r3}{$R_3$}
\psfrag{r4}{$R_4$}
\psfrag{r5}{$R_5$}
\psfrag{r6}{$R_6$}
\begin{center}
\includegraphics[scale=0.20]{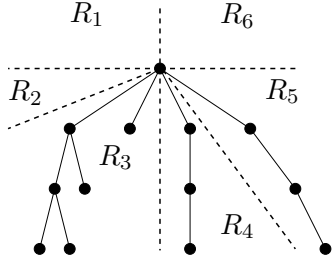}
\end{center}
\caption{The composition $\xi = (0,0,2,1,1,0)$ determines a placement
of the $4$ subtrees into $6$ regions.\label{fig:regions}}
\end{figure}

\begin{defn}\label{def:hscat}
Let $\sA^{(m)}_{s}$ be the set of pairs $(X,\xi)$, where $X$ is an
admissibly $m$-labeled plane tree with $mv+1$ vertices and $d$
subtrees, and $\xi $ is a composition of $d$ into $s$ parts.  We
define $C^{(m)}_{v,s} = |\sA^{(m)}_{s}|$ and put
\[
\sF^{(m)}_{s} (x)= \sum_{v\geq 0} C^{(m)}_{v,s}x^{v+1}. 
\]
\end{defn}

Clearly $\sF^{(m)}_{1}(x) = \sF^{( m)}(x)$, the generating function of
the hypergraph Catalans.  Here are some examples for $m=2$:
\begin{align*}
\sF^{(2)}_{1} (x) &= x + x^2 + 6x^3 + 57x^4 + 678x^5 + 9270x^6 +
139968x^7 + \dotsb ,\\
\sF^{(2)}_{2} (x) &= x + 3x^2 + 24x^3 + 267x^4 + 3546x^5 + 52938x^6 + 862974x^7 + \dotsb, \\
\sF^{(2)}_{3} (x) &= x + 6x^2 + 63x^3 + 834x^4 + 12672x^5 +
212436x^6 + 3854223x^7 + \dotsb, \\
\sF^{(2)}_{4} (x) &= x + 10x^2 + 135x^3 + 2130x^4 + 37320x^5 +
709560x^6+ 14472855x^7 + \dotsb .
\end{align*}

\subsection{} We will now compute the generating functions
$\sF^{(m)}_{s}$ by counting certain colored plane trees; this
generalizes results in \cite{gencat}.  Let $A (x) = \sum_{k\geq 0}
a_{k}x^{k}$ and $B (x) = \sum_{k\geq 0} b_{k}x^{k}$ be integral formal
power series with $a_{k}, b_{k}\geq 0$.  Let $\sP_{A,B}$ be the set of
all colored plane trees such that (a) any non-root vertex with $k$
children can be painted one of $a_{k}$ colors, and (b) if the root
vertex has degree $k$ then it can be painted any one of $b_{k}$
colors.  Let $P_{A,B}\in \ZZ \sets{x}$ be the ordinary generating
function of $\sP_{A,B}$.

For us, the ``coloring numbers'' will come from the generalized Wick
numbers and counting monomials: for any pair $t\geq 0, r\geq 1$, let
$\lambda (r,t)$ be the dimension of the space of degree $t$
homogeneous polynomials in $r$ variables.  We have
\[
\lambda (r,t) = \binom{r-1+t}{r-1}.
\]



\begin{thm}\label{thm:gf}
Put 
\[
\ell (x) = \sum_{d\geq 0} W_{m} (dm) \lambda (m,dm)x^{d}
\]
and 
\[
h_{s} (x) = \sum_{d\geq 0} W_{m} (dm) \lambda (m,sm) x^{d}.
\]
Then the set $\sP_{\ell, h_{s}}$ of $(\ell , h_{s})$-colored plane
trees is in bijection with the set of all admissibly $m$-labeled plane
trees $\sA^{(m)}_{s} $ whose subtrees have been placed into $s$
regions.  In particular we have $\sF^{(m)}_{s}  = P_{\ell , h_{s}}$.
\end{thm}

\begin{proof}
The proof is a simple modification of the proof of \cite[Theorem
5.6]{gencat} and follows easily from assembling facts in
Flajolet--Sedgewick \cite{fs};.  There we showed that there is a
bijection between admissibly $m$-labeled plane trees $\sA^{(m)}$ and
$(\ell ,h_{1})$-colored plane trees $\sP_{(\ell , h_{1})}$.  To extend
\cite[Theorem 5.6]{gencat} to the current setting, we need only
observe this bijection takes pairs $(X,\xi) \in \sA^{(m)}_{s}$ exactly
to $(\ell , h_{s})$-colored plane trees.  Indeed, the only new feature
is that in \cite{gencat} if the root of a plane tree had degree $d$,
we colored it by a set partition of $\sets{dm}$ into subsets of order
$m$, which gives $W_{m} (dm) $ colors and agrees with $h_{1}$.  Now,
however, we also color the root by a choice of composition $\xi$ of
$d$ into $s$ parts.  This is exactly accomplished by the coloring
series $h_{s} (x)$.
\end{proof}

\subsection{} To conclude this section, we explain why the tree series
$\sF_{s}^{(m)} (x)$ are needed to compute $\sG^{(m)}_{g}$ for larger
$g$.  The key is the following proposition, which shows how the series
appear when counting EDETs after grafting trees onto a fixed graph.

\begin{prop}\label{prop:countingedets}
Let $G$ be a graph with every vertex of degree at least $2$.  Then 
\[
\frac{\prod_{v\in V (G)}\sF^{(m)}_{md(v)}}{2me (G)|\Aut_{v}G|}\sum_{\gamma \in D
(G)} W (\gamma)M (\gamma) = \sum_{H}\frac{x^{v (H)}}{2me
(H)|\Aut_{v}H|}\sum_{\gamma_{H} \in D (H)} W (\gamma_{H}) M (\gamma_{H}),
\]
where the sum on the right is taken over all $H$ that can be built
from $G$ by grafting trees onto $V (G)$.
\end{prop}

\begin{proof}
Let $H$ be obtained from $G$ by grafting on trees.  Then there is a
bijection between $D (G)$ and $D (H)$.  Indeed, given any $\gamma \in
D (G)$ there is a unique balanced digraph structure $\gamma_{H}$ on
$\tilde{H}$ that extends $\gamma$, since the only balanced digraph
structure on a thickened tree is the canonical one.  Moreover, we have
$M (\gamma) = M (\gamma_{H})$, since the moment contributions of the
edges in $\gamma_{H} \smallsetminus \gamma$ all equal $1$.  Thus we
must show
\begin{equation}\label{eq:identity0}
\frac{W (\gamma )\prod_{v\in V (G)}\sF^{(m)}_{md(v)}}{2me
(G)|\Aut_{v}G|} = \sum_{H}\frac{W (\gamma_{H})x^{v (H)}}{2me
(H)|\Aut_{v}H|}.
\end{equation}

We fix $k\geq v (G)$ and consider the coefficients of $x^{k}$ on both
sides of \eqref{eq:identity0}.  Each graph $H$ on the right
corresponds to a tuple $\bT = (T_{1},\dotsc ,T_{v (G)})$ of rooted trees
with $\sum v (T_{i}) = k$.  Let $\arr{e}$ be a fixed directed edge of
$\gamma$, and let $\Upsilon  (\gamma)$ be the set of all Eulerian tours of
$\gamma$ with initial edge $\arr{e}$.  Let $\Upsilon  (\gamma_{H})$ be the
set of all Eulerian tours on $\gamma_{H}$ beginning with the same
edge.  For a rooted tree $T$ with root vertex $r$, let $\Upsilon _{0}
(T)$ be the set of all Eulerian tours on the canonical digraph
$\gamma_{T}$ with initial edge any directed edge with initial vertex
$r$. 

Given $H$ with corresponding tuple of trees $\bT = (T_{1}, \dotsc , T_{v
(G)})$, let $r_{i}$ be the root vertex of $T_{i}$.  Let $\sE
(\gamma_{H})$ be the set of Eulerian tours on $\gamma_{H}$.  Then we claim
there is a bijection
\begin{equation}\label{eq:Phi}
\Phi \colon \sE (\gamma _H) \longrightarrow \Xi (\bT )  \times \Upsilon  (\gamma) \times
\prod \Upsilon _{0} (T_{i}),
\end{equation}
where $\Xi (\bT ) $ is the set
\[
\Xi (\bT ) = \Bigl\{(\xi_{1},\dotsc ,\xi_{v (G)}) \Bigm| \text{$\xi_{i}$ is a
composition of $md (r_{i})$ into $md (v_{i})$ parts} \Bigr\}.
\]
The correspondence $\Phi$ is as follows.  Any Eulerian tour on
$\gamma_{H}$ starting with $\arr{e}$ determines a tour in $\Upsilon 
(\gamma)$ and tours in the $\Upsilon _{0} (T_{i})$: one simply keeps
records how edges are traversed in $\gamma$ and in the
$\gamma_{T_{i}}$.  This gives a map
\begin{equation}\label{eq:pi}
\pi \colon \Upsilon  (\gamma_{H})\longrightarrow \Upsilon  (\gamma )\times \prod \Upsilon _{0} (T_{i})
\end{equation}
that is clearly surjective.  Furthermore, given a tour $w\in \Upsilon 
(\gamma)$ and tours $w_{i}\in\Upsilon _{0} (T_{i})$, to assemble them into
a tour in $\Upsilon  (\gamma_{H})$ one needs to choose how much of $w_{i}$
to execute each time the vertex $v_{i}$ is entered in $w$.  This is
given by a composition $\xi_{i}$ of $md (r_{i})$ into $md (v_{i})$
parts.  Indeed, the outdegree of $r_{i}$ in $\gamma_{T_{i}}$ is $m d
(r_{i})$, and the indegree of $v_{i}$ in $\gamma$ is $md (v_{i})$.
Thus $v_{i}$ is visited $md (v_{i})$ times during $\gamma$, and the
total number of times we can exit $r_{i}$ to execute part of
$\gamma_{T_{i}}$ is $md (r_{i})$.  Thus in
\eqref{eq:pi} the preimages under $\pi$ correspond to the set $\Xi (\bT)$,
and \eqref{eq:Phi} is a bijection. 

Now we pass to EDETs.  The total number of Eulerian tours on $\gamma$
is $2me (G)|\Upsilon  (\gamma )|$.  Therefore 
\begin{equation}\label{eq:Wgamma}
W (\gamma) = 2me (G)\frac{|\Upsilon 
(\gamma )|}{\omega (\gamma)},
\end{equation}
where
\begin{equation}\label{eq:Gfactorials}
\omega (\gamma ) := \prod_{S\in \Ssingle \cup \Sparallel} a_{S}!b_{S}! \cdot \prod_{S\in
\Sloop} (2m|S|)!;
\end{equation}
in \eqref{eq:Gfactorials} the $a_{S}, b_{S}$ are the digraph parameters of $\gamma$, and
$\Ssingle$, $\Sparallel$, $\Sloop$ are the edge classes of $G$.  If
$H$ corresponds to the tuple of trees $\bT$,
then to compute $W (\gamma_{H})$ we must divide $2m e (H)|\Upsilon  (\gamma_{H})|$
by \eqref{eq:Gfactorials} as well as by
\[
J (\bT ) := \prod_{i=1}^{v (G)} ( 2me (T_{i}))!.
\]
Thus by \eqref{eq:Phi} we have
\begin{equation}\label{eq:WH}
W (\gamma_{H}) = 2me (H) \frac{|\Xi (\bT)|\cdot |\Upsilon  (\gamma)|}{J
(\bT) \cdot \omega (\gamma)}\prod_{i=1}^{v (G)} |\Upsilon _{0} (T_{i})|.
\end{equation}

Now we consider all $H$ contributing to the coefficient $x^{k}$, and
we incorporate automorphisms.  Since every vertex of $G$ has degree at
least $2$, we have $|\Aut H| = |\Aut G|
\prod |\Aut T_{i}|$.  Since trees have no loops or parallel
edges, we have $|\Aut_{v}H| = |\Aut_{v}G| \prod
|\Aut T_{i}|$.
Thus \eqref{eq:Wgamma} and \eqref{eq:WH} imply
\begin{equation}\label{eq:penultimate}
\sum_{\substack{H, v (H)=k}} \frac{W (\gamma_{H})}{2me (H)|\Aut_{v}H|}
= \frac{W (\gamma )}{2me (G) |\Aut_{v}G|}\sum_{\substack{\bT = (T_{1},\dotsc ,T_{v (G)})\\
\sum v (T_{i})=k}} \frac{|\Xi (\bT)|}{J (\bT)} \prod_{i=1}^{v (G)}
\frac{|\Upsilon _{0} (T_{i})|}{|\Aut T_{i}|}.
\end{equation}
To complete the proof, it remains to observe that the quantity
\[
\sum_{\substack{\bT = (T_{1},\dotsc ,T_{v (G)})\\
\sum v (T_{i})=k}} \frac{|\Xi (\bT)|}{J (\bT)} \prod_{i=1}^{v (G)}
\frac{|\Upsilon _{0} (T_{i})|}{|\Aut T_{i}|}
\]
in \eqref{eq:penultimate} is exactly the coefficient of $x^{k}$ in
$\prod_{v\in V (G)}\sF^{(m)}_{md(v)}$.
\end{proof}
\begin{remark}\label{rem:treeseries}
When $m=1$, the series $\sF^{(1)}_{1}$ is the ordinary generating
function of the Catalan numbers.  When $s>1$ the series are related to
\emph{self-describing sequences} due to \v{S}uni\'{k} \cite{sunik},
and to counting certain standard tableaux.
\end{remark}

\section{The series $\sG_{1}$}\label{s:genus1}

\subsection{} The goal of this section is to prove the following
theorem, which is the computation of the series $\sG_{1} (x)$ in terms
of the generalized tree functions.  Let
$\lk{n} (x)$ be the power series for $\log (1/ (1-x))$ with the terms
of degree $<n$ omitted:
\begin{align*}
\lk{n} (x) &= \log \bigl(\frac{1}{1-x}\bigr) - \sum_{k=1}^{n-1} \frac{x^{k}}{k}\\
&=  \frac{x^{n}}{n} + \frac{x^{n+1}}{n+1} +
\frac{x^{n+2}}{n+2} + \dotsb. \\
\end{align*}

\begin{thm}\label{thm:G1}
Let $\sF = \sF^{(m)}_{2m} (x)$.  Let $P$ be the set of pairs 
\[
P = \bigl\{(a,b) \in \ZZ^{2} \bigm| a,b \geq 0, a+b=2m, a-b \equiv 0 \bmod{4}\bigr\}.
\]
Let
\[
\beta (a) = \frac{a}{2m}\binom{2m}{a}.
\]
Then we have 
\begin{align*}
\sG_{1}^{(m)} (x) = \frac{\sF }{2m}  &+ \frac{\ipm{(z\bar
z)^{2m})}{m}\sF^{2}}{4m} \\ &+ \frac{1}{2m}\frac{(\beta (m)\sF)^{3}}{(1-\beta
(m)\sF)}+ \sum_{\substack{(a,b)\in P\\
 a\not =b}}\frac{1}{2 (a-b)}\lk{3} \Bigl(\frac{1-\beta (b)\sF}{1-\beta
(a)\sF}\Bigr),
\end{align*}
where the moment $\ipm{(z\bar
z)^{2m})}{m}$ is computed using the nonloop normalization.
\end{thm}

We note that the moment $\ipm{(z\bar
z)^{2m})}{m}$ appears in the second term in the sum, but no explicit
moments appear anywhere else. This occurs because the $2$-cycle is the
only cycle with a pair of parallel edges.  Similar explicit moments
occur in the expressions in
\S\S\ref{s:genusgparallel}--\ref{s:genusgloops}.

\subsection{}
The proof of Theorem \ref{thm:G1} occupies the rest of this section.
By \S \ref{ss:g1}, any graph $H$ with $g=1$ is a $k$-cycle $C_{k}$
with rooted trees grafted onto its vertices, which we denote $T_{1}$,
\dots , $T_{k}$.  We consider each of the quantities that appear in
the formula \eqref{eq:Gexp}.  We shall see that the description is
mostly uniform in $k$, although $k=1,2$ are exceptional.

\subsection{} First we consider the group $\Aut_{v} C_{k}$.  The group
$\Aut C_{k}$ has order $2k$ for all $k$, but only for $k\geq 3$ does
$|\Aut_{v}C_{k}|=2k$.  If $k=1$ we have $|\Aut_{v} (C_{1})| = 1$ since
the only nontrivial automorphism of $C_{1}$ fixes the vertex.  If
$k=2$ we have $|\Aut_{v} C_{2}| = 2$, since permuting the two edges
does not affect vertices.  To summarize,
\[
|\Aut_{v}C_{k}| = \begin{cases}
1 & \text{if $k=1$,}\\
2& \text{if $k=2$,}\\
2k & \text{if $k\geq 3$}.
\end{cases}
\]

\subsection{}
Next we consider balanced digraphs $D (H)$.  We first note that this
reduces to constructing balanced digraphs on $\tilde{C}_{k}$, since
the only balanced digraph structure on a thickened tree is the
canonical one.  We claim that if $k=1$ or $k=2$ there is only one
balanced digraph structure on $\tilde{C}_{k}$. If $k=1$ this is obvious.
If $k=2$ we have two edges $e, e'$ between two vertices $v,v'$.  All
choices of digraph parameters give $2m$ edges running from $v$ to $v'$
and $2m$ running the opposite direction.  Thus there is only one
digraph when $k=2$. 

For $k\geq 3$ any balanced digraph $\gamma$ on the thickening
$\tilde{H}$ is obtained as follows:

\begin{itemize}
\item On each $\tilde{T}_{i}$, we take the canonical
digraph structure.
\item On $\tilde{C}_{k}$, we fix $a, b\geq 0$ with $a+b=2m$, and then
put $a_{\{e \}} =a$, $b_{\{e \}}=b$ for all $e\in E (C_{k})$.
\end{itemize}

With this in hand, we can compute the contribution $M (\gamma)$ in
\eqref{eq:Gexp}.  
We have
\[
M(\gamma) = \begin{cases}
1 & \text{if $k=1$,}\\
\ipm{(z\bar z)^{2m})}{m}& \text{if $k=2$,}\\
1 & \text{if $k\geq 3$ and $a-b\equiv 0 \bmod 4$,}\\
0 & \text{otherwise}.
\end{cases} 
\]
where for $k=2$ we use the nonloop normalization.  Let $P$ be the set
of nonnegative pairs $(a,b)\in \ZZ^{2}$ satisfying $a+b=2m$,
$a-b\equiv 0 \bmod 4$.  We denote the digraph structure on
$\tilde{C}_{k}$ corresponding
to $(a,b)$ by $\gamma (a,b)$.

\subsection{}\label{} Now we must compute the number of EDETs on
$\gamma (a,b)$.  For $k=1$, there is a unique EDET on $\gamma
_{C_{1}}$.  For $k=2$, there are $2$ EDETs on $\gamma_{C_{2}}$,
corresponding to the choice of start vertex.  For $k\geq 3$, we use
the BEST theorem to count Eulerian tours, then divide by the
permutations of the edges to count EDETs.   Fix a
vertex $z$ of $C_{k}$.  The number of oriented spanning trees of
$\gamma (a,b)$ with root $z$ is given by

\begin{equation}\label{eq:taudef}
\sum_{i=0}^{k-1} a^{i}b^{k-1-i} = \begin{cases}
(a^{k}-b^{k})/ ( a-b) & \text{if $a\not =b$,}\\
km^{k-1} & \text{if $a=b=m$,}
\end{cases}
\end{equation}
since any spanning tree of a cycle is built by deleting an edge, and
since all digraph structures have $a$ edges running in one direction
and $b$ running in the opposite direction.
The sum \eqref{eq:taudef} will appear often in the sequel; we denote
it by $\tau (k; a,b)$.

The outdegree of each vertex in $\gamma (a,b)$ is always $2m$, since
the vertices of $C_{k}$ have degree $2$.  Since there are $k$
vertices, the product of factorials in the BEST theorem is therefore
$(2m-1)!^{k}$.  The total number of oriented edges in $\gamma (a,b)$
is $2mk$.  To pass to EDETs, we must divide by permutations of the
edges, which amounts to dividing by $a!^{k}b!^{k}$.  The result is
\[
\sum_{(a,b)\in P} W (\gamma (a,b))= 2mk
(2m-1)!^{k}\Bigl(\frac{km^{k-1}}{m!^{2k}}+\sum_{a\not =b}
\frac{a^{k}-b^{k}}{a!^{k}b!^{k} (a-b)}\Bigr).
\] 

\subsection{} Finally we can put these expressions together to
complete the proof.  Recall that we are writing $\sF^{(m)}_{2m}$ as
$\sF $. Using Proposition \ref{prop:countingedets}, the expression
\eqref{eq:Gexp}, and the computations above, we have
\begin{align*}
\sG_{1}^{(m)} (x) = \frac{\sF }{2m}  &+ \frac{\ipm{(z\bar
z)^{2m})}{m}}{4m}\sF^{2} \\
&+ \sum_{k\geq
3}\frac{\sF^{k}}{2mk}\frac{2mk(2m-1)!^{k}}{2k}\sum_{(a,b)\in P} \Bigl(
\frac{km^{k-1}}{m!^{2k}}+ \sum_{a\not =b}
\frac{a^{k}-b^{k}}{a!^{k}b!^{k} (a-b)} \Bigr). 
\end{align*}
The sum over $k$ can be simplified as follows.  Writing
\begin{equation}\label{eq:beta}
\beta (a) = \frac{a}{2m}\binom{2m}{a},
\end{equation}
we have
\[
\frac{1}{2}\sum_{k\geq 3} \frac{1}{k}\Bigl(\frac{k}{m}(\beta (m)\sF )^{k}+\sum_{\substack{(a,b)\in P\\
a\not =b}} \frac{(\beta (a)\sF )^{k}- (\beta (b)\sF )^{k}}{a-b}\Bigr).
\]
Interchanging the sum over $k$ with the sum over $(a,b)$
and using the definition of $\lk{n}$ completes the proof 
of Theorem \ref{thm:G1}.

\section{The series $\sG_{g}$ for $g\geq 2$: preliminaries}\label{s:genusgprelim}

\subsection{} Now we pass to general $g\geq 2$.  Given $G\in
\Gamma_{\geq 3} (g)$, let $D (G)$ be the set of balanced digraph
structures on the thickening $\tilde{G}$.  Let $\Sigma (G)$ be the set
of spanning trees of $G$.  We put an equivalence relation on $\Sigma
(G)$ by $T\sim T'$ if (i) $T$ and $T'$ have edges in $\Sparallel (G)$
and (ii) $T$ is obtained from $T'$ by permuting edges in a given
parallel class $S\in \Sparallel (G)$.  We fix once and for all a set
$\Sigma_{0} (G)\subset \Sigma (G)$ of representatives for the tree
classes.

For each $\gamma \in D (G)$ and $T\in \Sigma_{0} (G)$, we will define a
power series $R (\gamma, T)\in \QQ \sets{x}$ using rational functions
in the tree series $\sF^{(m)}_{s} (x)$.  Like the walks and
moment contributions in \S \ref{s:coeffs}, $R (\gamma , T)$ will be built as a product over
the edge classes $\cS (G)$.  Our final formula for $\sG^{(m)}_{g} (x)$
will take the shape
\begin{equation}\label{eq:Ggprelim}
\sG_{g}^{(m)} (x) = \sum_{G\in \Gamma_{\geq 3} (g)} \sum_{\gamma\in D
(G)}\sum_{T\in \Sigma_{0} (G)} R (\gamma , T).
\end{equation}
Each sum in \eqref{eq:Ggprelim} is taken over a finite set, which thus
yields a finite expression for our generating function in terms of
rational functions in tree functions.

\subsection{} Our strategy will be the same as in \S \ref{s:genus1},
except from the fact that sums over spanning trees and digraph
structures will remain.  Let $G'$ be a graph obtained from $G\in
\Gamma_{\geq 3} (g)$ by subdividing edges with new vertices of degree
$2$.  We encode $G'$ by \emph{subdivision parameters}
\[
\{k_{e}\in \ZZ \geq 0 \mid e \in E (g) \},
\]
where $k_{e}$ is the number of new vertices added to $e$ (thus $e$ has
been subdivided into $k_{e}+1$ new edges in $G'$).  Each $S \in \cS
(G)$ will have one or more subdivision parameters attached to it.  We
will compute how the quantities $W, M$ depend on the subdivision
parameters, as well as any adjustments to $|\Aut G|$ necessary to
compute $|\Aut_{v}G|$.  At the very end we apply Proposition
\ref{prop:countingedets} and insert tree functions for the vertices,
to account for grafting trees onto $G'$, and sum over the subdivision
parameters to obtain rational functions.

For the convenience of the reader, we summarize the notation/terminology we will
use when constructing $R (\gamma ,T)$:
\begin{itemize}
\item $G\in \Gamma_{\geq 3} (g)$.
\item $z$ is a fixed vertex on $G$.
\item $\gamma$ is a balanced digraph on $\tilde{G}$.
\item $\{a_{S}, b_{S} \mid S\in \Ssingle \cup \Sparallel \}$ is the
set of digraph parameters for $\gamma$ with respect to some fixed
orientations $\arr{e}$ of the edges; all edges in a parallel edge
class have the same orientation.  The parameter $a_{S}$ counts the
edges in $S$ that point in the same direction as $\arr{e}$, and
$b_{S}$ counts the edges that point in the opposite direction.
\item $T\in \Sigma_{0} (G)$ is a spanning tree of $G$.  By abuse of
notation we will also consider $T$ to be an oriented spanning tree, by
orienting each edge such that $z$ is the root.  If we need to
distinguish between $T$ and $T$ equipped with this orientation, we
write $\arr{T}$ for the latter.
\item Suppose $e$ appears in $T$ and $e\in S$.  Then either
$\arr{e}$ or $-\arr{e}$ (the same edge with opposite orientation)
appears in $\arr{T}$. We denote by $c_{S}$ the digraph parameter
distinguished by the orientation of $e$ in $\arr{T}$.  More precisely,
\[
c_{S} = \begin{cases}
a_{S} & \text{if $\arr{e}\in \arr{T}$,}\\
b_{S} & \text{if $-\arr{e}\in \arr{T}$}.
\end{cases}
\]
\item $R (\gamma ,T)$ will be given as product of factors $R_{\bullet}
(S, \gamma , T)$, where $\bullet \in \{\mathrm{s}, \mathrm{p},
\mathrm{l} \}$.  These three cases are covered in \S\S
\ref{s:genusgsingle}--\ref{s:genusgloops}.  We will give their
ingredients in the following order:
\begin{itemize}
\item [(\textbf{A})] contributions to correct automorphism counts (adjusting $|\Aut
G|$ to $|\Aut_{v} (G)|$);
\item [(\textbf{M})] contributions to $M (\gamma)$;
\item [(\textbf{O})] contributions to the outdegree product in the BEST theorem;
\item [(\textbf{E})] factors needed to pass from Eulerian tours to EDETs;
\item [(\textbf{F})] tree function factors arising from subdividing;
and
\item [(\textbf{T})] contributions to the spanning tree counts in $G'$.
\end{itemize}
\end{itemize}

We conclude this section by stating a lemma relating spanning trees in
$G$ to those in a subdivision $G'$.  We omit the simple proof.

\begin{lem}\label{lem:tree}
Let $G$ be a graph and let $e\in E (G)$.  Let
$G'$ be obtained from $G$ by subdividing $e$ into two edges $e_{1},
e_{2}$ by adding a degree 2 vertex.   Let $i\colon G-e \rightarrow
G'$ be the inclusion.  Let $T$ be a spanning tree of $G$.
\begin{enumerate}
\item If $e\in T$, then $i (T-e) \cup \{e_{1}, e_{2} \}$ is a spanning
tree of $G'$.
\item If $e\not \in T$, then $i (T) \cup \{e_{1} \}$ and $i (T)\cup
\{e_{2} \}$ are distinct spanning trees of $G'$. 
\end{enumerate}
Moreover, all spanning trees of $G'$ arise in this way, starting with
a spanning tree of $G$. \hfill \qed
\end{lem}

\section{The series $\sG_{g}$ for $g\geq 2$: single
edges}\label{s:genusgsingle}

\subsection{} First we consider single edges.  Let $S = \{e \}\in
\cS_{\text{s}} (G)$.  Since there is a unique edge in $S$, we drop the
$e$ from the subscripts of subdivision and digraph parameters.  Thus
we write $k, a, b, c$ for $k_e, a_S, b_S, c_S$.  In particular in $G'$ the
edge $e$ has been subdivided into a set $E$ of $k+1$ new edges.
  Note that in the
digraph $\gamma_{G'}$ each edge in $E$ becomes $a$ edges directed one
way, and $b$ edges directed the other.  Figure \ref{fig:eandE} shows
$e$ and $E\subset E ( G')$, as well as the corresponding part of
$\gamma (G')$ when $m=4$.

\begin{figure}[htb]
\psfrag{e}{$e$}
\psfrag{E}{$E$}
\psfrag{ab}{$\phantom{a}$}
\begin{center}
\includegraphics[scale=0.20]{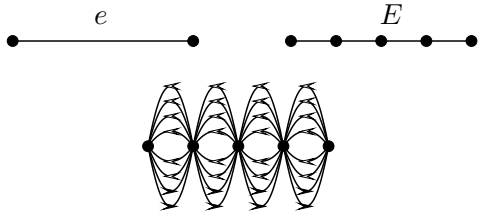}
\end{center}
\caption{$m=4$, $a=6$, $b=2$.\label{fig:eandE}}
\end{figure}

(\textbf{A})
There is no correction to the automorphism counts, since any
automorphism that acts nontrivially on $e$ will act nontrivially on
the new vertices in $G'$.

(\textbf{M})
After subdivision, we have $k+1$ groups of $a$ edges running in one
direction and $b$ edges running in the other, with $a+b=2m$.  Thus in
the computation of $M (\gamma)$, each of these groups contributes
$\ipm{z^{a}\bar z^{b}}{m}$, which equals $1$ if $a-b$ is divisible by
$4$, and is $0$ otherwise.

(\textbf{O}) Now we consider the outdegree contribution.  We have
divided $e$ into $k+1$ edges with all new vertices having degree $2$.
Thus in $\gamma_{G'}$ each new vertex will have outdegree $2m$.  This
means we need a factor of $(2m-1)!^{k}$ for the BEST theorem.

(\textbf{E})
Next we have the correction term to pass from Eulerian tours to
EDETs.  In $\gamma_{G'}$ we have $k+1$ groups of $a$ edges running in
one direction and $b$ edges in the other.  These can be permuted among
themselves, which gives a factor of $1/ (a!b!)^{k+1}$.

(\textbf{F})
Next we have the tree functions for the new vertices in $G'$; this
gives $(\sF_{2m}^{(m)})^{k}$.

(\textbf{T}) Finally we consider counting spanning trees.  If $e\in T$,
then by Lemma \ref{lem:tree}  we must select one edge from each group to
build a tree in $G'$.  Thus there are $c^{k+1}$ choices.

If $e\not \in T$, then to build a spanning tree of $\gamma_{G'}$ we
must first pick an edge of $E$ to delete, and then make choices from
the directed edges to point back at the endpoints of $e$.  The total
contribution to the number of spanning trees is therefore
\[
a^{k} + a^{k-1}b + \dotsb + a{b^{k-1}} + b^{k} = \tau (k+1; a,b).
\]
\subsection{}
Combining these contributions together, we obtain
\[
\sum_{k\geq 0} A (k) \frac{(2m-1)!^{k}}{a!^{k+1}b!^{k+1}} (\sF^{(m)}_{2m})^{k},
\]
where $A (k) = c^{k+1}$ or $\tau (k+1; a,b)$ depending on whether
$e\in T$ or $e\not \in T$.  If we sum over $k\geq 0$, we obtain the
contribution $\Rsingle (S,\gamma ,T)$.  To give the final result, and
to simplify later formulas, we introduce some notation.  Recall the
function $\beta$ defined in \eqref{eq:beta}.  Assume
$a+b=2m$, and let $K\geq 0$.  Then let
\begin{equation}\label{eq:mudef}
\mu_{\geq K} (c) = \frac{c (\beta (c) \sF_{2m}^{(m)})^{K}}{a!b!(1-\beta (c) \sF_{2m}^{(m)})};
\end{equation}
if $4 \mid (a-b)$; otherwise put $\mu_{\geq K} (c)=0$.
Here $c\in \{a,b \}$, and will be unambiguously selected by the context
in which $\mu$ appears.  Also let
\begin{equation}\label{eq:phidef}
\varphi_{\geq K} (a,b) = \begin{dcases}\frac{1}{a!b!(a-b)}\biggl(\frac{a (\beta (a)\sF_{2m}^{(m)})^{K}}{1-\beta (a)\sF_{2m}^{(m)}} -
\frac{b (\beta (b)\sF_{2m}^{(m)})^{K}}{1-\beta (b)\sF_{2m}^{(m)}}\biggr)&\text{if $a\not =b$,}\\
\frac{1}{m!^{2}\bigl( 1-\beta
(m)\sF_{2m}^{(m)}\bigr)^{2}}-\frac{1}{m!^{2}}\sum_{i=1}^{K} i (\beta (m)\sF_{2m}^{(m)})^{i-1}&\text{if $a=b=m$,}
\end{dcases}
\end{equation}
if $4\mid (a-b)$; otherwise put $\varphi_{\geq K} (a,b)=0$.  

\begin{defn}\label{def:noloopsparallel}
Let $S = \{e \}\in \Ssingle (G)$ be a edge class consisting of a
single edge.  Let $\gamma \in D (G)$, and $T\in \Sigma_{0}
(G)$.  We define  
\begin{equation}\label{eq:Rsingle}
\Rsingle  (S,\gamma ,T) = \begin{cases}
\mu_{\geq 0} (c)& \text{if $e\in T$,}\\
\varphi_{\geq 0} (a,b) &\text{otherwise},
\end{cases}
\end{equation}
where $a=a_{S},b=b_{S}$ are the digraph parameters for $S$.
\end{defn}

\section{The series $\sG_{g}$ for $g\geq 2$: parallel edges}\label{s:genusgparallel}

\subsection{} Next we consider $S$ a set $\{e_{1},\dotsc ,e_{n} \}$ of
$n$ parallel edges.  The digraph $\gamma$ determines 
parameters $a_{S}, b_{S}$ with $a_{S}+b_{S}=2m|S|$.  

Suppose that in passing to $G'$ we insert subdivision points in the
first $r$ edges $e_{1},\dotsc ,e_{r}$ (Figure \ref{fig:parallel}).
Let $k_{1},\dotsc ,k_{r}$ be the corresponding subdivision parameters.
After subdividing, we must consider the possible digraph structures on
the thickening $\tilde{G}'$.  We have $r+1$ new digraph parameter
pairs $a_{1},b_{1},\dotsc ,a_{r+1},b_{r+1}$.  The pairs $a_{i},
b_{i}$, $i=1,\dotsc ,r$ correspond to the subdivided edges
$e_{1},\dotsc ,e_{r}$, and the final pair $a_{r+1},b_{r+1}$
corresponds to the remaining $n-r$ unsubdivided edges.  These
parameters record how the flow through the set $S$ is distributed
among the first $r$ edges and the remaining $n-r$.  The new parameters
must be nonnegative and must satisfy the system
\begin{align}
a_{i} + b_{i} &= 2m\quad \text{for $i=1,\dotsc ,r$,}\label{eq:Dr0}\\
a_{r+1}+b_{r+1} &= 2m (n-r),\\
\sum a_{i} &= a_{S}, \quad \sum b_{i} = b_{S}.\label{eq:Dr1}
\end{align}
Note that if $r=n$, we have $a_{n+1} = b_{n+1}=0$.

\begin{figure}[htb]
\psfrag{v}{$v$}
\psfrag{vp}{$v'$}
\begin{center}
\includegraphics[scale=0.15]{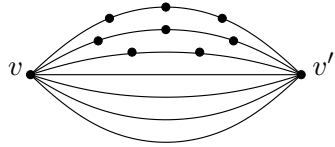}
\end{center}
\caption{Subdividing a group of parallel edges.  Here $n=7$, $r=3$,
and the subdivision parameters are $3,3,2$.\label{fig:parallel}}
\end{figure}

Let $D (\gamma ,r)$ be the finite set of integral solutions to this
system.  We call this set the \emph{local digraph parameters} of $S$.
Our final contribution will include a sum over $D (\gamma ,r)$.  Now
we consider the various contributions.

\textbf{(A)}  The automorphism correction is $(n-r)!$, since the edges
$e_{r+1},\dotsc ,e_{n}$ have no subdivision points, and thus any
permutation of them does not move any vertices.

\textbf{(M)}  The total contribution to the moments is 
\[
\prod_{i=1}^{r+1} \ipm{z^{a_{i}}\bar z^{b_{i}}}{m}.
\]
This product vanishes unless $a_{i}-b_{i}$ is divisible by $4$ for all
$i$.  If this holds, then in fact $\ipm{z^{a_{i}}\bar z^{b_{i}}}{m} =
1$ for $i=1,\dotsc ,r$.

\textbf{(O)} Each new vertex has degree $2$ in $G'$.  Thus the
outdegree contribution to the BEST theorem is 
\[
\prod_{i=1}^{r} (2m-1)!^{k_{i}}.
\]

\textbf{(E)}
The ambiguity factor in passing from Eulerian tours to EDETs comes
from permuting the last $n-r$ edges and the new edges obtained after
subdivision.  Thus it equals 
\[
a_{r+1}!b_{r+1}!\prod_{i=1}^{r} (a_{i}!b_{i}!)^{k_{i}+1}.
\]

\textbf{(F)}
We have a tree function $\sF_{2m}^{(m)}$ for each new vertex in $G'$.
Thus we obtain
\[
\prod_{i=1}^{r}(\sF^{(m)}_{2m})^{k_{i}}.
\]

\textbf{(T)}
Finally we consider the spanning trees in $\tilde{G}'$.  There are two
cases: (i) some $e_{i}$ was in the spanning tree $T$ for $G$, and (ii)
no $e_{i}$ was in $T$. 

For (i), let the original endpoints of the $e_{i}$ be $v, v'$.  By
Lemma \ref{lem:tree} we must use some $e_{i}$, or its new edges after
subdivision, to make a spanning tree in $G'$.  Suppose we use $e_{i}$
with $i\leq r$.  Then there are $c_{i}^{k_{i}+1}$ possible choices for
edges to connect $v,v'$ after subdividing $e_{i}$, and $\tau (k_{j}+1;
a_{j}, b_{j})$ choices for the other subdivided edges $e_{j}$, $j\not
= i$.  If we instead use an edge from the unsubdivided ones $e_{i}$,
$i\geq r+1$, then we have $c_{r+1}$ choices to connect $v,v'$ and
$\tau (k_{i}+1; a_{i}, b_{i})$ choices for $i=1,\dotsc ,r$.  The total
number is therefore
\begin{equation}\label{eq:einT}
c_{r+1}\prod_{i=1}^{r} \tau (k_{i}+1; a_{i},b_{i}) \\
+ \sum_{i=1}^{r} c_{i}^{k_{i}+1}\prod_{\substack{j\leq r\\
j\not =i}}\tau (k_{j}+1; a_{j}, b_{j}).
\end{equation}
Note that if $r=n$, then $c_{n+1}=0$ since it is either $a_{n+1}$ or
$b_{n+1}$, and the first term in \eqref{eq:einT} vanishes.

For (ii), we don't connect $v, v'$ in this case.  The total number of
trees is thus 
\begin{equation}\label{eq:enotinT}
\prod_{i=1}^{r}\tau (k_{i}+1; a_{i}, b_{i}).
\end{equation}

\subsection{} Now we combine all these contributions together by
summing over the subdivision parameters $k_{i}$ and using that we have
$\binom{n}{r}$ choices for the subset of $r$ edges to be
subdivided. To not overcount, we must sum over $k_{i}\geq 1$.  This
leads to the following.

\begin{defn}\label{def:parallel}
Let $S=\{e_{1},\dotsc ,e_{n} \}\in \Sparallel (G)$.  Let $\gamma \in D
(G)$ and $T\in \Sigma_{0} (G)$.  Given $0 \leq r \leq n$, let $D
(\gamma , r)$ be the digraph structures given by the parameters
\eqref{eq:Dr0}--\eqref{eq:Dr1} .
\begin{enumerate}
\item Suppose no $e_{i}$ is in the spanning tree $T\in \Sigma_{0} (G)$.  Then we put
\[
\Rparallel (S,\gamma ,T) = \sum_{r=0}^{n}\sum_{D (\gamma
,r)}\frac{n!}{r!}\cdot \frac{\ipm{z^{a_{r+1}}\bar
z^{b_{r+1}}}{m}}{a_{r+1}!b_{r+1}!}\prod_{i=1}^{r} \varphi_{\geq 1}
(a_{i},b_{i}).
\]
\item Suppose some (and therefore exactly one) $e_{i}\in T$.  Then we put
\begin{multline*}
\Rparallel (S,\gamma ,T) = \sum_{r=0}^{n}\sum_{D (\gamma
,r)}\frac{n!}{r!}\cdot
\frac{\ipm{z^{a_{r+1}}\bar z^{b_{r+1}}}{m}}{a_{r+1}!b_{r+1}!}\\
\cdot \Biggl(
c_{r+1} \prod_{i\leq r} \varphi_{\geq 1} (a_{i},b_{i})+ \sum_{i=1}^{r} \mu_{\geq 1} (c_{i}) \prod_{\substack{j\leq r\\
j\not =i}}\varphi_{\geq 1} (a_{j},b_{j})\Biggr).
\end{multline*}
\end{enumerate}
\end{defn}

\section{The series $\sG_{g}$ for $g\geq 2$: loops at a vertex}\label{s:genusgloops}

\subsection{}
Finally we consider the edge classes $S \in \Sloop (G)$.  Suppose $S$
consists of $n$ loops $e_{1},\dotsc ,e_{n}$ at a vertex $v$.  Let
$(r,s,t)$ be a composition of $n$.  Suppose the subdivision parameters
$k_{1}, \dotsc , k_{n}$ satisfy (i) $k_{i}\geq 2$ for $1\leq i \leq
r$, (ii) $k_{i} = 1$ for $r+1 \leq i \leq r+s$, and (iii) $k_{i}=0$
for $i>r+s$ (Figure \ref{fig:loops}).

As with the parallel edges, we have a set of local digraph parameters
after subdivision.  Fix an arbitrary orientation of the $e_{i}$.  The
first $r$ edges have parameters $a_{i}, b_{i}\geq 0$, $i=1,\dotsc ,r$,
with $a_{i}+b_{i}=2m$.  The next $s$ edges have a unique digraph
structure: $a_{i}=m, b_{i}=m$ for $i=r+1, \dotsc , r+s$.  The last $t$
edges each become $m$ loops in $\tilde{G}'$.  We let $D (\gamma ,
r,s)$ be this set of local digraph parameters $\{a_{i}, b_{i} \mid
i=1,\dotsc ,r+s \}$.

\begin{figure}[htb]
\begin{center}
\includegraphics[scale=0.15]{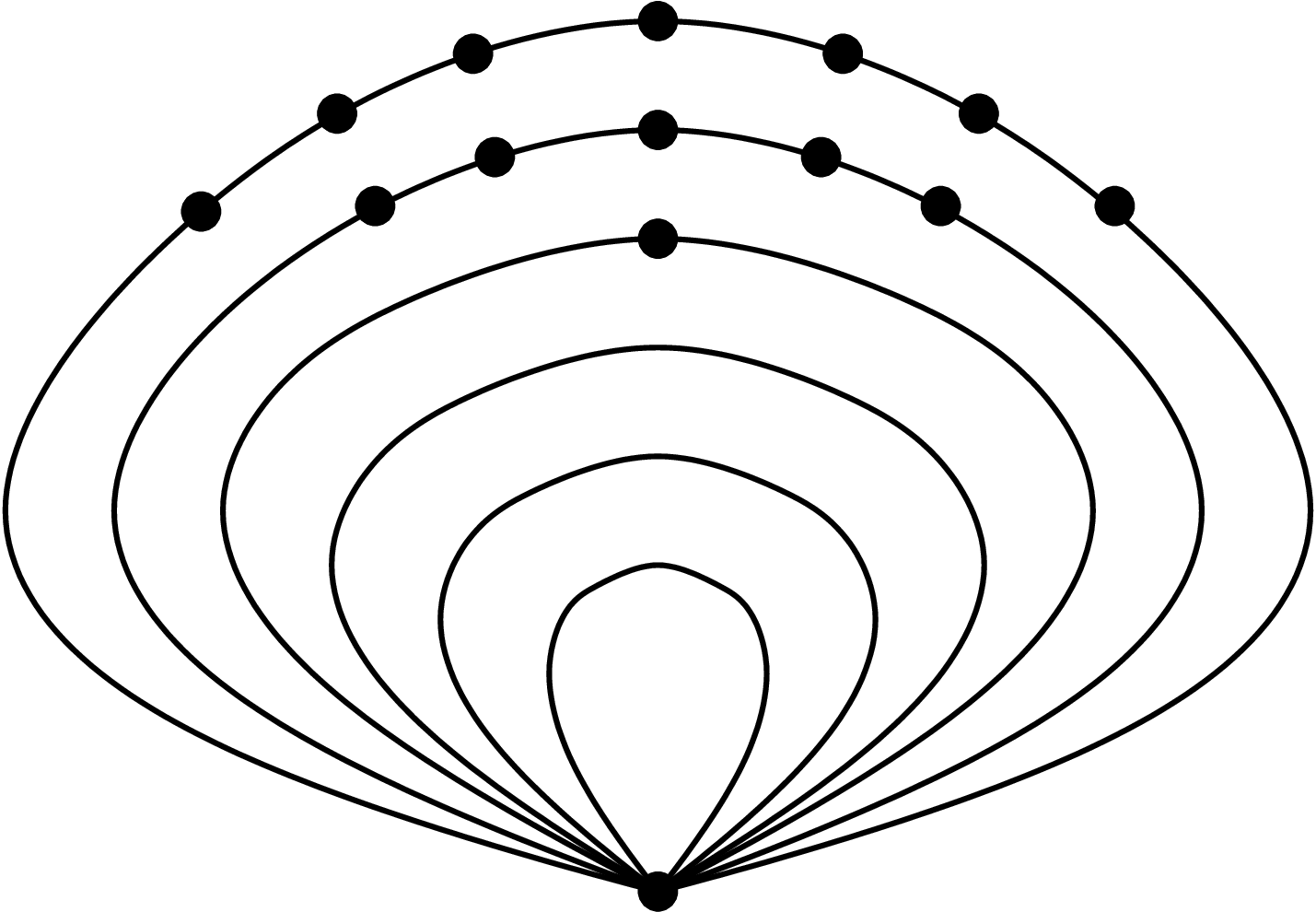}
\end{center}
\caption{Subdividing a set of loops.  Here $n=6$, $r=2$, $s=1$, $t=3$,
and the subdivision parameters are $7, 5, 1$.\label{fig:loops}}
\end{figure}

\textbf{(A)}
Any permutation or flip of the first $r$ loops moves the new vertices.
Permutations of the next $s$ loops move vertices, but flipping them
does not.  Neither permutations nor flips of the last $t$ loops affect
any vertices.  Thus the automorphism correction factor is $2^{s+t}t!$.

\textbf{(M)}  The total contribution to the moments is 
\begin{equation}\label{eq:mprod}
\ipm{x^{2mt}}{m}\cdot \prod_{i=1}^{r} \ipm{z^{a_{i}}\bar z^{b_{i}}}{m} \cdot
\prod_{i=r+1}^{r+s} \ipm{(z\bar z)^{2m}}{m}.
\end{equation}
The first moment uses the loop normalization; all others use the
nonloop normalization.
The product \eqref{eq:mprod} vanishes unless $a_{i}-b_{i}$ is divisible by $4$ for
$i=1,\dotsc ,r$.  If this holds, then in fact $\ipm{z^{a_{i}}\bar
z^{b_{i}}}{m} = 1$ for $i=1,\dotsc ,r$.  

\textbf{(O)} Each new vertex has degree $2$ in $G'$.  Thus the
outdegree contribution to the BEST theorem is 
\[
\prod_{i=1}^{r+s} (2m-1)!^{k_{i}}.
\]

\textbf{(E)}
The ambiguity factor in passing from Eulerian tours to EDETs comes
from permuting the new edges in the first $r$ loops, the two groups of
$2m$ edges for the next $s$ loops, and the last $2mt$ loops.  Thus it equals 
\[
(2mt)!\cdot (2m)!^{2s} \cdot \prod_{i=1}^{r} (a_{i}!b_{i}!)^{k_{i}+1}.
\]

\textbf{(F)}
We have a tree function $\sF_{2m}^{(m)}$ for each new vertex in $G'$.
Thus we obtain
\[
(\sF^{(m)}_{2m})^{s}\prod_{i=1}^{r}(\sF^{(m)}_{2m})^{k_{i}}.
\]

\textbf{(T)}  No $e_{i}\in S$ can appear in any $T\in \Sigma_{0}
(G)$, but there will be nontrivial contributions to spanning trees
after subdividing, since the new vertices need to be reached.  The
first $r$ loops give $\tau (k_{i}+1; a_{i}, b_{i})$ for $i=1,\dotsc
,r$.  The next $s$ loops each give $2m$, since the spanning trees must
point back to the vertex $z$, which is a fixed vertex in $G$. Finally
the last $t$ loops do not appear in any spanning tree.

\subsection{} Now we combine all these contributions together by
summing over the subdivision parameters $k_{i}$, and use that we have
$\binom{n}{r\phantom{a}s\phantom{a}t}$ choices for the $r$, $s$, and
$t$ loops, to obtain the following:

\begin{defn}\label{def:loops}
\[
\Rloop (S,\gamma , T) = \sum_{\substack{(r,s,t)\\
r,s,t\geq 0\\
r+s+t=n}} 2^{s+t}t!\binom{n}{r\phantom{a}s\phantom{a}t}
\frac{\ipm{x^{2mt}}{m}}{(2mt)!}\frac{\ipm{(z\bar z)^{2m}}{m}^{s}}{(2m)!^{s}}(\sF_{2m}^{(m)})^{s}\sum_{D (\gamma ,r,s)}\prod_{i=1}^{r}\varphi_{\geq 2} (a_{i},b_{i})
\]
\end{defn}

\section{The main theorem}\label{s:main}
Finally we put all this together to prove our main result.

\begin{thm}\label{thm:main}
Let $G\in \Gamma_{\geq 3} (g)$ with $g\geq 2$.  Let $d (v)$ be the
degree of $v\in V (G)$.  Let $\Sigma_{0} (G)$
be the set of representatives of spanning trees of $G$.  Let $D (G)$
be the balanced digraph structures on the thickening $\tilde{G}$.  For
each $\gamma \in D (G)$ and $T\in \Sigma_{0} (G)$, put
\begin{multline*}
R (\gamma ,T) = \frac{1}{|\Aut G|}\prod_{v\in V (G)}\sF^{(m)}_{md (v)}(md (v) -1)! \\
\times \prod_{S\in \Ssingle } \Rsingle (S, \gamma ,T)
\prod_{S\in \Sparallel} \Rparallel (S, \gamma ,T) \prod_{S\in
\Sloop } \Rloop  (S, \gamma ,T),
\end{multline*}
where the functions $\Rsingle , \Rparallel , \Rloop$ are respectively given in
Definitions \ref{def:noloopsparallel}, \ref{def:parallel}, \ref{def:loops}.
Then 
\begin{equation}\label{eq:Gthm}
\sG_{g}^{(m)} (x) = \sum_{G\in \Gamma_{\geq 3} (g)} \sum_{\gamma\in D
(G)}\sum_{T\in \Sigma_{0} (G)} R (\gamma , T).
\end{equation}
\end{thm}

We remark that we do not know an analogue of the technique in Remark
\ref{rmk:genfuns} to compute $\sG_{g}^{(m)}$ for large $g$ without
enumerating all of $\Gamma_{\geq 3} (g)$.

\begin{proof}
By Corollary \ref{cor:Gg}, we have
\begin{equation}\label{eq:Gexp2}
\sG_{g}^{(m)}(x) = \sum_{H\in
\Gamma (g)} \frac{x^{v (H)}}{2me (H)|\Aut_{v} H|} \sum_{\gamma_{H} \in D (H)} W (\gamma_{H})M (\gamma_{H}),
\end{equation}
since $e (H) = v (H)+g-1$.  According to
Proposition \ref{prop:countingedets} and its proof, we can rewrite
this as
\begin{equation}\label{eq:deg2}
\sG_{g}^{(m)} (x) = \sum_{G\in \Gamma_{\geq 2} (g)}\frac{\prod_{v\in V (G)}\sF^{(m)}_{md(v)}}{2me (G)|\Aut_{v}G|} \sum_{\gamma \in D (G)}W (\gamma)M (\gamma).
\end{equation}
Note that the sum \eqref{eq:deg2} is taken over $\Gamma_{\geq 2} (g)$;
we have only removed the grafted trees from \eqref{eq:Gexp2}, not the
vertices of degree 2.  We can thus rewrite \eqref{eq:deg2} as 
\begin{equation}\label{eq:deg3}
\sG_{g}^{(m)} (x) = \sum_{G\in \Gamma_{\geq 3} (g)}\prod_{v\in V
(G)}\sF^{(m)}_{md(v)} \sum_{G'} \frac{\prod_{v\in V (G')\smallsetminus V (G)}\sF_{2m}^{(m)}}{2me (G')|\Aut_{v}G'|}\sum_{\gamma \in D (G')}W (\gamma)M (\gamma),
\end{equation}
where the second sum is taken over all subdivisions $G'$ of $G$.  We
claim that when the rational functions $\Rsingle$, $\Rparallel$, and $\Rloop$ are
unfolded into infinite series, we exactly obtain the right hand side
of \eqref{eq:deg3}.  Indeed, the terms in the product of these
infinite series correspond exactly to the subdivisions $G'$ of $G$
(including the bi-/multinomial coefficients in $\Rparallel$,
$\Rloop$).  Each of the quantities \textbf{(A)},\dots ,\textbf{(T)} in
\S \S \ref{s:genusgsingle}--\ref{s:genusgloops} was constructed to
achieve this identity:
\begin{itemize}
\item  The automorphism factors \textbf{(A)} were chosen so that
their product with $1/|\Aut G|$ produces $1/|\Aut_{v} G'|$.
\item  Similarly the moment factors \textbf{(M)} guarantee that we
have the correct value of $M (\gamma)$ for any $\gamma \in D (G')$.
\item The products of tree functions \textbf{(F)} give a factor of
$\sF_{2m}^{(m)}$ for each vertex in $V (G')\smallsetminus V (G)$.
\item The local digraph parameters in $\Rparallel$, $\Rloop$ guarantee
that we have accounted for all digraph structures in $D (G')$ coming
from a given structure in $D (G)$.
\item Finally we count EDETs $W (\gamma)$ for $\gamma \in D (G')$ by
first using the BEST theorem (Theorem \ref{thm:best}) to count
Eulerian tours and then dividing by factorials to account for
permuting parallel edges and loops.  The factors \textbf{(O)} provide
the outdegree factors for the vertices in $V (G')\smallsetminus V
(G)$, and the factors \textbf{(T)} count the oriented spanning trees
that come from $\Sigma_{0} (G)$, using Lemma \ref{lem:tree}.  The
products of factorials \textbf{(E)} allow us to pass from Eulerian
tours to EDETs.  In fact, in our application of the BEST theorem we
only count Eulerian tours starting with a fixed directed edge, not all
Eulerian tours.  Consequently when passing to EDETs we actually
compute $W (\gamma)/2me (G')$, not $W (\gamma)$.
\end{itemize}

This shows that the right of \eqref{eq:Gthm} equals \eqref{eq:deg3},
and completes the proof. 
\end{proof}

\section{Examples and complements}\label{s:example}

\subsection{}\label{ss:singleedgesexample} We begin by showing how to
use Definition \ref{def:noloopsparallel} to compute the contribution
of any $G\in \Gamma_{\geq 3} (g)$ with no loops and no parallel edges.
We illustrate with the complete graph $G = K_{4}$, which has $g=3$.
According to Theorem \ref{thm:main}, we see that $G$ contributes
\[
\sK_4^{(m)} (x) := \frac{\prod_{v\in V (G)}(md (v) -1)!\sF^{(m)}_{md (v)} }{|\Aut G|} 
 \sum_{\gamma \in D (G)}\sum_{T\in \Sigma (G)} \,\prod_{e\in E (G)} \Rsingle (\{e \}, \gamma ,T)
\]
in the sum over $\Gamma_{\geq 3} (g)$.

Each vertex of $K_{4}$ has degree $3$, and $|\Aut G| = 24$.  The digraph
structures that give nonzero expectations for $m\geq 1$ are counted by
\[
1,15,15,65,65,175,175,369,\dotsc .
\]
The distinct numbers in this list are called the \emph{rhombic
dodecahedral numbers}, and equal $n^{4}- (n-1)^{4}$, $n\geq 1$
\cite[Series \texttt{A005917}]{oeis}.

We consider $m=2$.  The thickening $\tilde{G}$ has 4 parallel edges
for each edge in $G$.  Up to symmetry, the
15 digraphs are the following:
\begin{itemize}
\item There are $8$ obtained by choosing a vertex $v$, orienting the
edges incident to $v$ as in the canonical digraph (two inedges, two
outedges), and then orienting the remaining edges opposite $v$ in one
of two different cycles (all edges pointing in the same direction).
\item There are $6$ obtained by choosing a pair of opposite edges
$e,e'$ in $G$.  The edges in $\tilde{G}$ corresponding to $e,e'$ are
oriented as in the canonical digraph, and the remaining edges are
oriented into one of two different cycles.
\item Finally we have the canonical digraph.
\end{itemize}

There are of course $16=4^{2}$ spanning trees in $G$ by Cayley's
theorem.  If we fix a vertex $z$, then each $T$ becomes an oriented
spanning tree $\arr{T}$ rooted at $z$.  A given $\arr{T}$ may not be
compatible with a given digraph structure $\gamma$, since $\gamma$ may
not have any edges oriented in the same direction as a given edge of
$T$ (the corresponding $c_{S}$ vanishes).  Note that for the canonical
digraph, \emph{any} $\arr{T}$ is compatible and has $c_{S} \not =0$
for all $S$.

Computing the various $\Rsingle (S,\gamma , T)$ is somewhat tedious by hand, but
is easily done on a computer.  One finds for the final series 
\[
\sK_{4}^{(2)} (x) = 566875x^{4}/2 + 31318750x^5 + 4134735625x^{6}/2 +
106250840000x^7 + \dotsb .
\]

\subsection{} For our next example we consider the complete series
$\sG_{2}^{(2)}$.  There are three graphs in $\Gamma_{\geq 3} (2)$,
shown in Figure \ref{fig:3graphs} together with the orders of their
automorphism groups.  We consider the contribution of each separately.

\subsection{}
The graph $G_{1}$ consists of $3$ parallel edges, so its contribution
is determined solely by $\Rparallel$.  We have 
\begin{equation}\label{eq:3pe}
525x^2/4 + 29325x^3/4 + 1136175x^4/4 + 9483525x^5 + 587571825x^6/2 +
\dotsb . 
\end{equation}

We can analyze \eqref{eq:3pe} to see how it is assembled from
subdividing and grafting.  First we claim the leading term comes from
$G_{1}$ with no subdivision points and no trees attached.  Indeed,
there is a unique digraph structure on the thickening $\tilde{G}_{1}$
($6$ edges running from left to right, and $6$ running from right to
left).  Consequently there are $2$ EDETs.  The moment contribution is
$\ipm{z^{6}\bar z^{6}}{2} = 1575$, the order of $\Aut_{v} G_{1}$
is $2$, and the denominator $2mr$ in \eqref{eq:weighting} is $12$.
Putting all this together, we find that $G_{1}$'s contribution to
$\sG_{2}^{(2)}$ is $1575x^{2}/12 = 525x^{2}/4$.  By Proposition \ref{prop:countingedets},
the series
\begin{multline}\label{eq:3pe2}
{525}(\sF_{6}^{(2)})^{2}/4 \\
= 525x^2/4 + 11025x^3/2 +
694575x^4/4 + 9955575x^5/2 + 551922525x^6/4 + \dotsb 
\end{multline}
then gives the contribution to $\sG_{2}^{(2)}$ of \emph{all} graphs
built from $G_{1}$ by grafting trees onto its vertices, but with no
subdivision of any of its edges.

Subtracting \eqref{eq:3pe2} from \eqref{eq:3pe} we obtain
\begin{equation}\label{eq:3pe3}
7275x^3/4 + 110400x^4 + 9011475x^5/2 + 623221125x^6/4 + \dotsb.
\end{equation}
We claim \eqref{eq:3pe3} gives the contribution of the graphs built by
grafting trees onto the subdivision $G'$ of $G_{1}$ gotten by adding a
new vertex to an edge.  Indeed, one can also check directly that the
contribution of $G'$ is $7275x^{3}/4$, although this computation is
more involved: there are now three different digraph structures to
consider.

Proposition \ref{prop:countingedets} implies ${7275} \sF_{4}^{(2)}
(\sF_{6}^{(2)})^{2}/4$ gives the total contribution of $G'$ to
$\sG_{2}^{(2)}$; subtracting it from \eqref{eq:3pe3}, gives
\[
15825x^4 + 1090125x^5 + 194316225x^6/4 + \dotsb, 
\]
which is the total contribution of grafting trees onto two different
subdivisions of $G_{1}$, one with 2 new vertices on an edge, and one
with 2 vertices added to 2 different edges.  One can continue to
subtract multiples of products of tree functions to isolate the
contributions of more complicated subdivision types.

\subsection{}\label{ss:g22} Next we consider the graph $G_{2}$.  We will be briefer.
Its contribution comes only from $\Rloop$, and is
\begin{multline}\label{eq:g2}
35x/8 + 1295x^2/4 + 117985x^3/8 \\
+ 1104635x^4/2 + 150908485x^5/8 
+ 2467763795x^6/4 
+ \dotsb. 
\end{multline}
The first contribution is from $G_{2}$ itself, and is $W_{4} (8) = 35$
divided by $8$, the number of edges in the thickening.  This is the
only graph with one vertex in $\Gamma_{\geq 3} (2)$, and thus gives
the leading term of $\sG_{2}^{(2)}(x)$.  In fact it's clear that a
similar phenomenon happens for all $g\geq 2$ and all $m$: the leading
coefficient of $\sG_{g}^{(m)} (x)$ is $W_{2m} (2mg)/(2mg)$.

\subsection{} Finally we have the graph $G_{3}$.  Its contribution
comes from factors of type $\Rsingle$ and $\Rloop$, and is
\begin{equation}\label{eq:g3}
25x^2/4 + 2075x^3/4 + 102575x^4/4 + 2010125x^5/2 + 138990425x^6/4 + \dotsb . 
\end{equation}
Adding together \eqref{eq:3pe}, \eqref{eq:g2}, and \eqref{eq:g3} we find
\begin{multline}\label{}
\sG_{2}^{(2)} (x) 
=  35x/8 + 1845x^2/4 \\
+ 180785x^3/8 + 862005x^4 +
234817185x^5/8 + 1890948935x^6/2 
+ \dotsb .
\end{multline}

\begin{figure}[htb]
\psfrag{12}{$|\Aut G_{1}| = 12$}
\psfrag{a8}{$|\Aut G_{2}| = 8$}
\psfrag{b8}{$|\Aut G_{3}| = 8$}
\begin{center}
\includegraphics[scale=0.15]{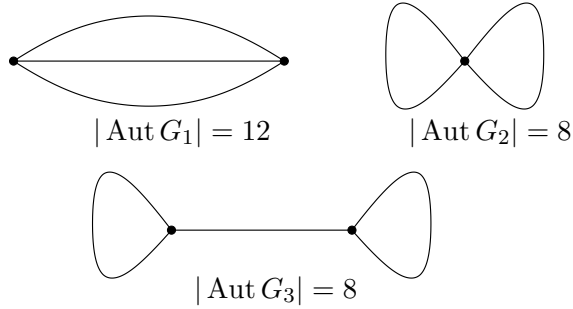}
\end{center}
\caption{The three graphs in $\Gamma_{\geq 3} (2)$.\label{fig:3graphs}}
\end{figure}

\subsection{}
Next we consider the classical case of the GUE ($m=1$) in more
detail.  The polynomials $\sP^{(1)}_{2r} (N)$ can be easily computed
using a generating function due to Harer--Zagier (cf.~\cite[\S 3.1]{lz}); 
The first few are (cf.~\eqref{eq:hzpolys})
\begin{multline*}
\sP^{(1)}_{2} (N) = N^{2}/2, \quad \sP^{(1)}_{4}(N) = N^{3}/2+N/4,\\
 \sP^{(1)}_{6} (N) = 5N^{4}/6+5N^{2}/3, \quad 
 \sP^{(1)}_{8} (N) =
7N^5/4 + 35N^3/4 + 21N/8.
\end{multline*}
The generating functions $\sG_{g}^{(1)} (x)$, $g\leq 3$ are
\begin{align*}
\sG_{0}^{(1)} &= x^2/2 + x^3/2 + 5x^4/6 + 7x^5/4 + 21x^6/5 + 11x^7 + \cdots,\\
\sG_{1}^{(1)} &= x/2 + 3x^2/2 + 5x^3 + 35x^4/2 + 63x^5 + 231x^6 + 858x^7 + \cdots,\\
\sG_{2}^{(1)} &= 3x/4 + 15x^2/2 + 105x^3/2 + 315x^4 + 3465x^5/2 + 9009x^6 + 45045x^7 + \cdots,\\
\sG_{3}^{(1)} &= 5x/2 + 105x^2/2 + 630x^3 + 5787x^4 + 45297x^5 + 318447x^6 + 2072040x^7 + \cdots.\\
\end{align*}
One can check that these do reproduce the polynomials $\sP^{(1)}$,
after passing from the falling factorial basis to the usual monomial
basis:
\begin{align*}
(N)_{2}/2 + (N)_{1}/2 &= N^{2}/2,\\
(N)_{3}/2 + 3(N)_{2}/2+3 (N)_{1}/4 &= N^{3}/2+N/4,\\
5(N)_{4}/6 + 5 (N)_{3} + 15 (N)_{2}/2 + 5 (N)_{1}/2 &=5N^{4}/6+5N^{2}/3,\\
7 (N)_{5}/4 + 35 (N)_{4}/2 + 105 (N)_{3}/2 + 105 (N)_{2}/2 + 105
(N)_{1}/8 &= 7N^5/4 + 35N^3/4 + 21N/8.
\end{align*}
The computation of the last polynomial used the first coefficient of
$\sG_{4}^{(1)}$, which we know (\S \ref{ss:g2}) equals $W_{2} (8)/8$.

\subsection{}\label{ss:schemes} We remark that, even though the polynomials $\sP^{(1)}$
are defined in terms of counting weighted orientable unicellular maps,
the series $\sG_{g}^{(1)}$ include contributions from graphs that can
only appear in \emph{nonorientable} unicellular maps, such as $2$
vertices connected by two parallel edges.  

On the other hand, it is possible to give Wright-like expressions
enumerating orientable unicellular maps of fixed genus $h$.  Instead
of the sets $\Gamma_{\geq 3} (g)$, we take the finite set of genus $h$
orientable unicellular maps $\sM (h)$ that are the reductions of all
the orientable unicellular maps of genus $h$.  In particular their
$1$-skeleta will be a subset of $\Gamma_{\geq 3} (g)$ where $g=2h$.
In the literature this is known as the method of schemes, and is due
to Chapuy--Marcus--Schaeffer \cite{cms}.

For example, let $\sM (1) = \{M_{1}$, $M_{2}\}$ be the genus 1
orientable unicellular maps in Figure \ref{fig:m1}.  The orders of
their automorphism groups are $4$ and $6$ respectively. Their
$1$-skeleta are respectively the graphs $G_{1}, G_{2} \in \Gamma_{\geq
3} (2)$. These maps are the reductions of all genus $1$ orientable
unicellular maps, in the sense that any of the latter can be built
from $M_{1}, M_{2}$ by subdividing edges and grafting trees.  Just
like with the series $\sG_{g}$, the trees are placed into various
regions.  However this time we are actually physically placing trees
on an orientable surface, which means the correct tree function to use
is $\sF = \sF^{(1)}$, the ordinary generating function of the
Catalans:
\[
\sF = x + x^2 + 2x^3 + 5x^4 + 14x^5 + 42x^6 + \dotsb.
\]
After thinking about how the subdivision and placement works, the
pleasure of which we leave to the reader, one sees that the
Wright-like expression is
\begin{multline*}
\frac{x/(1-\sF)^{4}}{4\bigl(1-{x}/(1-\sF)^{2}\bigr)^{2}} +
\frac{x^{2}/(1-\sF)^{6}}{6\bigl(1-{x}/(1-\sF)^{2}\bigr)^{3}}
\\
= x/4 + 5x^2/3 + 35x^3/4 + 42x^4 + 385x^5/2+858x^{6}+\dotsb.
\end{multline*}
The coefficient of $x^{k}$ is the weighted count of genus 1 orientable
unicellular maps with $k$ vertices.  In particular it equals the
subleading coefficient of $\sP^{(1)}_{2r}(N)$, where $r = k+1$.

\begin{figure}[htb]
\psfrag{m1}{$M_1$}
\psfrag{m2}{$M_2$}
\begin{center}
\includegraphics[scale=0.4]{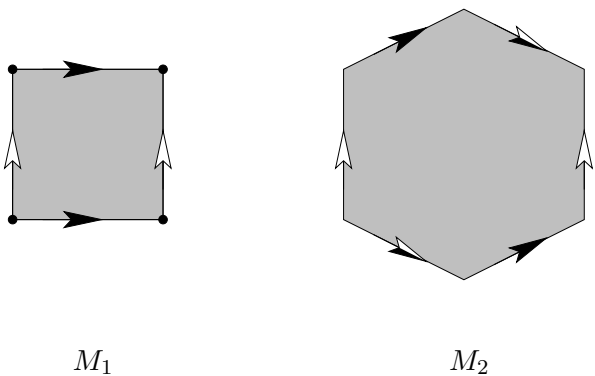}
\end{center}
\caption{The reduced genus 1 maps\label{fig:m1}}
\end{figure}

\bibliographystyle{amsplain_initials_eprint_doi_url}
\def\dddots{\mathinner{\mskip1mu\raise1pt\vbox{\kern7pt\hbox{.}}\mskip2mu
\raise4pt\hbox{.}\mskip2mu\raise7pt\hbox{.}\mskip1mu}}
\bibliography{gen_funs}

\newcommand{\jumpgap}{\newline\phantom{-}}
\begin{landscape}
\appendix
\section{Examples of $\sG_{g}$}\label{app}
\footnotesize
\begin{table}[htb]
\begin{center}
\begin{tabular}{|c||p{7.5in}|}
\hline
$m$ & $\sG^{(m)}_{0} (x)$\\
\hline
1 & $x^{2}/2 + x^{3}/2 + 5x^{4}/6 + 7x^{5}/4 + 21x^{6}/5 + 11x^{7} + 429x^{8}/14 + 715x^{9}/8 + 2431x^{10}/9 + O(x^{11})$ \\ 
2 & $x^{2}/4 + 3x^{3}/4 + 19x^{4}/4 + 339x^{5}/8 + 927x^{6}/2 + 5832x^{7} + 326439x^{8}/4 + 19935963x^{9}/16 + 41062743x^{10}/2 + O(x^{11})$ \\ 
3 & $x^{2}/6 + 5x^{3}/3 + 430x^{4}/9 + 7150x^{5}/3 + 178160x^{6} + 56884400x^{7}/3 + 2813812000x^{8} +  1712301278000x^{9}/3 + 4156291931860000x^{10}/27 + O(x^{11})$ \\ 
4 & $x^{2}/8 + 35x^{3}/8 + 5075x^{4}/8 + 3521875x^{5}/16 + 632774975x^{6}/4 + 216912857000x^{7} + 4147718813494375x^{8}/8 +  63445001652155046875x^{9}/32 \jumpgap + 45358226806367283484375x^{10}/4 + O(x^{11})$ \\ 
5 & $x^{2}/10 + 63x^{3}/5 + 9996x^{4} + 139108158x^{5}/5 + 5563080718146x^{6}/25 +  21304570653435636x^{7}/5 + 834024715253210846736x^{8}/5 \jumpgap + 11914956019396567890508896x^{9} +  7159118148880765211316108608784x^{10}/5 + O(x^{11})$ \\ 
6 & $x^{2}/12 + 77x^{3}/2 + 529837x^{4}/3 + 4236584737x^{5} + 389093885143656x^{6} +  103951883534149644840x^{7} + 66464625342098684554314384x^{8} \jumpgap + 89031583111338263998369110425544x^{9} +  226882507639908104856032696370170600496x^{10} + O(x^{11})$ \\ 
7 & $x^{2}/14 + 858x^{3}/7 + 23643048x^{4}/7 + 719818214544x^{5} + 761736881828317824x^{6} +  19890091205586992696537088x^{7}/7 \jumpgap + 1462480414853122505621647925600256x^{8}/49 +  5296882526955042374928855101040107624448x^{9}/7 \jumpgap +  289378206931080342008004282036570702635383799808x^{10}/7 + O(x^{11})$ \\ 
8 & $x^{2}/16 + 6435x^{3}/16 + 1093132755x^{4}/16 + 4189518834318675x^{5}/32 + 12789261471380923418175x^{6}/8 \jumpgap + 83653285789705777304380006500x^{7} + 232714435447974068218421872500830070375x^{8}/16 \jumpgap + 450580766581918742133057691058873988427995046875x^{9}/64 \jumpgap + 66577050254084689335011968052311347098762983586132124375x^{10}/8 + O(x^{11})$ \\ 
9 & $x^{2}/18 + 12155x^{3}/9 + 38865369400x^{4}/27 + 224396619941386750x^{5}/9 + 10568964372411801962927750x^{6}/3 \jumpgap + 23406053777968695174488548034562500x^{7}/9 + 22608292449374493728456232435833623726000000x^{8}/3 \jumpgap + 631377425706842492186064055388314507534328247515000000x^{9}/9 \jumpgap + 146415040193166445578321910204976829328192757969036532154956250000x^{10}/81 + O(x^{11})$ \\ 
10 & $x^{2}/20 + 46189x^{3}/10 + 156327779036x^{4}/5 + 24550017450057210683x^{5}/5 + 40288257357688473370014249901x^{6}/5 \jumpgap + 421795713069639073592252496895953739582x^{7}/5 + 20488501368379858937904964722489355369283655924016x^{8}/5 \jumpgap + 3689644411278439406922287219109006620041431766233435541621568x^{9}/5 \jumpgap + 2084651796127409275171394859267740680422750018266628726230222158866466184x^{10}/5 + O(x^{11})$ \\
11 & $x^{2}/22 + 176358x^{3}/11 + 7647037369608x^{4}/11 + 10915046033893459789968x^{5}/11 + 208825449554970891825780671394048x^{6}/11 \jumpgap + 2831839254360209065821483392498206940318208x^{7} + 2316892466620063091353698696530631770384310687210928128x^{8} \jumpgap + 8110630946423855711520713287693292112294656381212102639305749526528x^{9} \jumpgap + 100972721844729199137629840094010555618643839225533556449970936697396719873548288x^{10} + O(x^{11})$ \\ 
\hline
\end{tabular}
\medskip
\caption{The series $\sG_{0}^{(m)}$ for various $m$.}
\end{center}
\end{table}

\end{landscape}

\newpage

\footnotesize
\begin{landscape}
\begin{table}
\begin{center}
\begin{tabular}{|c||p{7.5in}|}
\hline
$m$ & $\sG^{(m)}_{1} (x)$\\
\hline
1 & $x/2 + 3x^{2}/2 + 5x^{3} + 35x^{4}/2 + 63x^{5} + 231x^{6} + 858x^{7} + 6435x^{8}/2 + 12155x^{9} + 46189x^{10} + O(x^{11})$ \\ 
2 & $x/4 + 39x^{2}/8 + 1057x^{3}/12 + 26185x^{4}/16 + 157754x^{5}/5 + 2525385x^{6}/4 + 367153025x^{7}/28  + 9053433989x^{8}/32 + 114374039159/18x^{9} \jumpgap + 2979865876233/20x^{10} + O(x^{11})$ \\ 
3 & $x/6 + 29x^{2} + 3243x^{3} + 1227709x^{4}/3 + 311618041x^{5}/5 + 35580025289x^{6}/3 + 20201052406319x^{7}/7  + 903625169254959x^{8} \jumpgap + 3240595862314433341x^{9}/9  + 2680799675238902615957x^{10}/15 + O(x^{11})$ \\ 
4 & $x/8 + 3895x^{2}/16 + 4059437x^{3}/24 + 5204433801x^{4}/32 + 1240396387983x^{5}/5 + 15154687179735595x^{6}/24  + 147190828334904053197x^{7}/56 \jumpgap + 1077501388739307715748101x^{8}/64   + 5613970557194552312344615747x^{9}/36  + 15914424081024558631497263410773x^{10}/8 + O(x^{11})$ \\ 
5 & $x/10 + 12564x^{2}/5 + 32083231x^{3}/3 + 82587122497x^{4} + 6979851926820641x^{5}/5  + 262524072181728393711x^{6}/5 \jumpgap + 138731580702016806662613751x^{7}/35  + 532488083344053052990052124595x^{8} + 1045672114389417544918043101760508973x^{9}/9  \jumpgap + 38604370319323054284363237685152615903579x^{10} + O(x^{11})$ \\ 
6 & $x/12 + 116837x^{2}/4 + 6847525345x^{3}/9 + 1165011392509523x^{4}/24 + 141460024012035797783x^{5}/15  + 64420778716593229928198347x^{6}/12 \jumpgap + 156683231905015971019299137035904x^{7}/21  + 1024391277584504875018600400508768726007x^{8}/48  \jumpgap + 3010014099065750064277607345553462822831314413x^{9}/27  \jumpgap + 19606282660270056806062522109584922512451138370254111x^{10}/20 + O(x^{11})$ \\ 
7 & $x/14 + 2546262x^{2}/7 + 1228809417760x^{3}/21 + 31444176646732790x^{4} + 71091608120734564592304x^{5}  \jumpgap + 1853394978669414931036436102234x^{6}/3 + 112191006703206431623585270856099225456x^{7}/7  \jumpgap + 6939041952914264555206031229184136910033518394x^{8}/7  + 7915971015992240473269406040825553417128428779825734512x^{9}/63  \jumpgap + 1035272404429542108448172149417469808754919047768041903490875754x^{10}/35 + O(x^{11})$ \\ 
8 & $x/16 + 151253943x^{2}/32 + 228052206504493x^{3}/48 + 1397757980876495400585x^{4}/64 + 5767548831242402514932611223x^{5}/10 \jumpgap + 1231275171181741396390748452591442873x^{6}/16 + 4231335846480059704899207618033982387673876365x^{7}/112 \jumpgap + 6548916174232091030686340350049859456809749215707147013x^{8}/128 \jumpgap + 11435195813394543804766839996340325324817416094081051762817250339x^{9}/72 \jumpgap + 80683653170933361641288958757680115742413259997515064267919166008414610601x^{10}/80 + O(x^{11})$ \\ 
9 & $x/18 + 189587090x^{2}/3 + 1203679216640107x^{3}/3 + 143900076497370594887809x^{4}/9 + 74021140007836384535720212635811x^{5}/15 \jumpgap + 91451889346448494192743569765800688050835x^{6}/9 + 2007225422752754477190344385451963291695748611296921x^{7}/21 \jumpgap + 2860173376372287661501850637902429303867448806185261728911323x^{8} \jumpgap + 5906373233896191330140592688800504889551709642615892062820915184806331391x^{9}/27 \jumpgap + 339030908816907996776778251832825089901455486724743899111578426248197089283390054123x^{10}/9 + O(x^{11})$ \\ 
10 & $x/20 + 17250869207x^{2}/20 + 523460023579698428x^{3}/15 + 97566016615792150922165901x^{4}/8 \jumpgap + 1099324174761245977554106941626941089x^{5}/25 + 16850317736297508032200277329067435621119926011x^{6}/12 \jumpgap + 1790451965882058489121620839500053547399376221398607556744x^{7}/7 \jumpgap + 13615008856320845898666158436970847713964618491055549249058627889123757x^{8}/80 \jumpgap + 14490562978875417260238042561242499064785143920194465253439868309052418023482540527x^{9}/45 \jumpgap + 150979534103674001261766717589363112749250051943439236066336480752222265418424457950481966690483x^{10}/100 + O(x^{11})$ \\ 
\hline
\end{tabular}
\medskip
\caption{The series $\sG_{1}^{(m)}$ for various $m$.}
\end{center}
\end{table}

\end{landscape}

\newpage

\footnotesize
\begin{landscape}
\begin{table}
\begin{center}
\begin{tabular}{|c||p{7.5in}|}
\hline
$m$ & $\sG^{(m)}_{2} (x)$\\
\hline
1 & $3x/4 + 15x^2/2 + 105x^3/2 + 315x^4 + 3465x^5/2 + 9009x^6 + 45045x^7 + 218790x^8 + 2078505x^9/2 + 4849845x^{10} + O(x^{11})$ \\
2 & $35x/8 + 1845x^2/4 + 180785x^3/8 + 862005x^4 + 234817185x^5/8 + 1890948935x^6/2 + 237793731875x^7/8 \jumpgap + 930399023100x^8 + 117663720672365x^9/4 + 1900266044070095x^{10}/2 + O(x^{11})$ \\
3 & $77x/2 + 72562x^2 + 30009350x^3 + 10280579400x^4 + 3658261039160x^5 + 1491680067031588x^6 + 729844653641699526x^7 \jumpgap + 433489372122070371240x^8 + 309745940001212295052780x^9 + 261891076428894172851532340x^{10} + O(x^{11})$ \\
4 & $6435x/16 + 165046365x^2/8 + 1112593168665x^3/16 + 236663849715480x^4 + 17600341065798248505x^5/16 \jumpgap + 31078789546809537473985x^6/4 + 1319481703145262519322964955x^7/16 + 2476799604846061873047750387285x^8/2 \jumpgap + 199090249189138121971283755029243765x^9/8 + 2574962587392198618100308020091493760655x^{10}/4 + O(x^{11})$ \\
5 & $46189x/10 + 16671436353x^2/2 + 223860132037399x^3 + 7648026953950977594x^4 + 485230365078411716232612x^5 \jumpgap + 319493821749765037584889653974x^6/5 + 15657595846461128720675122036809094x^7 \jumpgap + 6176229944803737214210334186152831066722x^8 + 3579777621181045714314704486350520500645242404x^9 \jumpgap + 2893488391090598220413674334197875035971303358739346x^{10} + O(x^{11})$ \\
6 & $676039x/12 + 4020649596600x^2 + 2648282038452752611x^3/3 + 305073077224518147820457x^4 + 263445965041708572528771592590x^5 \jumpgap + 1987690460788987576093939555431724381x^6/3 + 11575564605880212320228689951412973889953823x^7/3 \jumpgap + 40796928744360103714744030383601651967209396789201x^8 \jumpgap + 2076809590266914583719656282023706333521336331576677089768x^9/3 \jumpgap + 53131191274484186284816220649209868804416052294470109218750485069x^{10}/3 + O(x^{11})$ \\
7 & $5014575x/7 + 14853865945093200x^2/7 + 3944396285733144972000x^3 + 14080257416242165101623191200x^4 \jumpgap + 164136655070860919228052973306996800x^5  + 7963567689382365008448110010645698384079600x^6 \jumpgap + 1123699102842309195689283902392972144152904236100800x^7 \jumpgap + 2256110625863372547880902839761751171507560102609828206168000x^8/7 \jumpgap + 1129961539059001970451043695947559095072311179494408722337211966984000x^9/7 \jumpgap + 131695805305487185246452181778558041000191649379126019772533309318461244982800x^{10} + O(x^{11})$ \\
8 & $300540195x/32 + 18878669392183676829x^2/16 + 610241819764163714517332121x^3/32 + 722364922604422998315893355104460x^4 \jumpgap + 3623863437352070918534948180462182007331897x^5/32 + 850390376622185287573118983757464079688347889380961x^6/8 \jumpgap + 11783054048078221264040880277065942943188562080960725818533339x^7/32 \jumpgap + 11533639759284784434770233447749832483540533472461690938388807315293493x^8/4 \jumpgap + 685170839905386560975202160107133286451901755827294134999606804325599793221394613x^9/16 \jumpgap + 8960619686783079974829603212713949975870289157392370136878832982656658838355556246845031759x^{10}/8 + O(x^{11})$ \\
9 & $756261275x/6 + 4076281886503589420075x^2/6 + 291491340669707075823290582525x^3/3 + 40071867596945571422276932811026392750x^4 \jumpgap + 253857304535751141356364561021090304858539561500x^5/3 \jumpgap + 4604105999972157532616141037507813252196108750629194217850x^6/3 \jumpgap + 131666983676987844755762656268864651107767021468074027572815728288750x^7 \jumpgap + 28243534384710028911252564511951832660850228392373269469253220235604502274505050x^8 \jumpgap + 37438622728343342360386241936370078370994606975743901557826986346905373626733949158174965300x^9/3 + O (x^{10})$ \\
\hline
\end{tabular}
\medskip
\caption{The series $\sG_{2}^{(m)}$ for various $m$.}
\end{center}
\end{table}

\end{landscape}

\end{document}